\documentclass{article}
\usepackage{graphicx} 
\usepackage{graphicx} 
\usepackage{a4wide}
\usepackage[T1]{fontenc}
\usepackage{a4wide}
\usepackage{lipsum}
\usepackage{mathtools}
\usepackage{amsthm}
\usepackage{amsfonts}
\usepackage{yfonts}
\usepackage{amssymb}
\usepackage{xcolor}
\usepackage{mathrsfs}
\usepackage{amsmath}
\usepackage{cite}
\usepackage{epsfig}
\usepackage{stmaryrd}
\usepackage[utf8]{inputenc} 
\usepackage[shortlabels]{enumitem}
\usepackage[T1]{fontenc}
\usepackage{caption}
\usepackage{soul}

\newcommand\NN{{\mathbb N}}
\newcommand\ZZ{{\mathbb Z}}

\usepackage[utf8]{inputenc}
\usepackage[utf8]{inputenc}
\usepackage{fullpage}
\usepackage{amsthm}
\usepackage[pdftex,colorlinks,citecolor=blue]{hyperref}
\usepackage{hyperref} 
\usepackage{cleveref}
\usepackage{amsmath, amssymb, amsbsy}
\usepackage{algorithm}
\usepackage{algcompatible}
\usepackage{algpseudocode}
\usepackage{tikz,caption}
\usetikzlibrary{shapes,arrows,calc}

\allowdisplaybreaks

\crefname{figure}{figure}{}
\Crefname{figure}{Figure}{}

\usetikzlibrary{patterns,calc}

\newcommand{\beq}[1]{\begin{equation}\label{#1}}
\newcommand{\enq}[0]{\end{equation}}

\newcommand{\se}{\subseteq}

\newcommand{\cC}{\mathcal{C} }

\newcommand{\cK}{\mathcal{K} }

\newtheorem{theorem}{Theorem}[section]
\newtheorem{problem}{Problem}
\newtheorem{conjecture}[problem]{Conjecture}

\newtheorem{lemma}[theorem]{Lemma}
\AddToHook{env/lemma/begin}{\crefalias{theorem}{lemma}}
\newtheorem*{claim*}{Claim}

\AddToHook{env/claim/begin}{\crefalias{theorem}{claim}}
\newtheorem{proposition}[theorem]{Proposition}
\AddToHook{env/proposition/begin}{\crefalias{theorem}{proposition}}
\newtheorem{corollary}[theorem]{Corollary}
\AddToHook{env/corollary/begin}{\crefalias{theorem}{corollary}}
\newtheorem{fact}[theorem]{Fact}
\AddToHook{env/fact/begin}{\crefalias{theorem}{fact}}
\newtheorem*{fact*}{Fact}

\theoremstyle{definition}

\newtheorem{defn}[theorem]{Definition}
\AddToHook{env/defn/begin}{\crefalias{theorem}{definition}}
\newtheorem{remark}[theorem]{Remark}
\AddToHook{env/remark/begin}{\crefalias{theorem}{remark}}
\newtheorem*{remark*}{Remark}

\newenvironment{subproof}[1][\proofname]{%
  \begin{proof}[#1]%
}{%
  \end{proof}%
}

\definecolor{andrey}{rgb}{0.3, 0.65,,0.2}
\definecolor{trees}{rgb}{1, .7,.5}
\definecolor{vertex}{rgb}{.65, 0,.85}

\title{A linear upper bound for zero-sum Ramsey numbers of bounded degree graphs}
\author{Jasmin Katz\thanks{London School of Economics. Email: \tt{j.katz3@lse.ac.uk}} \and Xiaopan Lian\thanks{Center for Combinatorics and LPMC, Nankai University, Tianjin 300071, 
China. E-mail: {\tt{Lian@nankai.edu.cn.}}Research supported by National Natural Science Foundation of China No. 12371351.} \and Alexandru Malekshahian\thanks{Mathematical Institute, University of Oxford. E-mail: {\tt{alex.malekshahian@maths.ox.ac.uk}}.  Research supported by ERC Advanced Grant 883810.} \and Andrey Shapiro\thanks{King's College London. E-mail: \tt{andrey.shapiro@kcl.ac.uk}}}
\date{}

\begin{document}
\maketitle
 \begin{abstract}
 Let $G$ be a graph and $\Gamma$ a finite abelian group. The zero-sum Ramsey number of $G$ over $\Gamma$, denoted by $R(G, \Gamma)$, is the smallest positive integer $t$ (if it exists) such that for any edge-coloring $c:E(K_t)\rightarrow \Gamma$, there exists $G'\subset K_t$ isomorphic to $G$ so that $\sum_{e\in E(G')} c(e)=0_\Gamma$, where $0_\Gamma$ is the identity of $\Gamma$.

 We prove a linear upper bound $R(G, \Gamma)\leq Cn$ that holds for every $n$-vertex graph $G$ with bounded maximum degree and every finite abelian group $\Gamma$ with $|\Gamma|$ dividing $e(G)$.

    \end{abstract}

\section{Introduction}

Ramsey theory is a central subject in combinatorics which has very recently seen a number of impressive developments. Tracing its origins to the work of Frank Ramsey in 1930 \cite{ramsey}, the (classical) Ramsey number of a graph $G$ is the minimum $n\in\mathbb{N}$ such that any red-blue coloring of the edges of the complete graph $K_n$ contains a monochromatic copy of $G$. Ramsey's theorem states that $R(G)$ exists for all $G$. 

In this work we are interested in a sub-area of Ramsey theory called \emph{zero-sum} Ramsey theory. This has its origin in the celebrated Erd\H{o}s-Ginzburg-Ziv theorem \cite{EGZ}, stating that given a set of $2m-1$ elements in $\mathbb{Z}_m$, there exists a subset of $m$ elements summing to zero. This was generalized by Olson \cite{OLSON1976}, who showed that the same statement holds with any group (not necessarily abelian) in place of $\ZZ_m$. Motivated by these results, the following notion of zero-sum Ramsey numbers was introduced by Bialostocki and Dierker \cite{BD92}.  Recall that the \emph{exponent} $\exp(\Gamma)$ of a finite abelian group $\Gamma$ is the smallest positive integer $m$ such that $g\cdot m=e$ for all $g\in \Gamma$; for a finite abelian group this is the same as the least common multiple of the orders of all the elements of $\Gamma$.

\begin{defn}\label{def:zero_sum}
    Given a (finite) graph $G$ and a finite abelian group $(\Gamma, +)$ such that $\exp(\Gamma)$ divides $e(G)$, the \emph{zero-sum Ramsey number} of $G$ over $\Gamma$, denoted by $R(G, \Gamma)$,  is the minimum integer $t$ such that for any edge-coloring $c:E(K_t)\rightarrow \Gamma$, there exists $G'\subset K_t$ isomorphic to $G$ so that $\sum_{e\in E(G')} c(e)=0_\Gamma$, where $0_\Gamma$ is the identity of $\Gamma$.
\end{defn}
It is easy to see that $R(G, \Gamma)\leq R_{|\Gamma|}(G)$, where $R_r(G)$ represents the $r$--color (classical) Ramsey number of $G$, and therefore $R(G, \Gamma)$ is finite. A monochromatic coloring of $E(K_t)$ with a non-zero element of $\Gamma$ also shows that the requirement for $\exp(\Gamma)$ to divide $e(G)$ is necessary for this quantity to be finite. 

One of the early successes of zero-sum Ramsey theory was the exact determination of $R(G, \mathbb{Z}_2)$ for all graphs $G$ by Caro \cite{Caro94}. This is one of the few known cases where the zero-sum Ramsey numbers are known exactly. Other, more recent cases that have been settled are $R(F, \mathbb{Z}_3)$ when $F$ is a forest  \cite{ADP25} and $R(F, \mathbb{Z}_2^2)$ when $F$ is an odd forest \cite{ACPS22}; see also the survey \cite{Carosurvey} for a broad overview of known results. Given the difficulty in determining these quantities exactly, it is natural to instead ask for estimates, and the majority of work in this field has indeed focused on proving upper bounds \cite{AC93,CD25,Caro92}.

In fact, \Cref{def:zero_sum} was initially stated for cyclic groups $\Gamma$, and the majority of literature on the subject also focuses on this case. Notable exceptions include \cite{ACPS22}, where the groups $\mathbb{Z}_2^d$ were investigated, or \cite{CKMS25, CGHS25} and the references therein, which provide upper bounds for the existence of \emph{some} zero-sum cycle, without any control over its length, over $\mathbb{Z}_p^d$ and over general groups of odd order, respectively.

Another notable direction was explored by Caro \cite{Caro92allgroups}, where in analogy to Olson's generalization of the Erd\H{o}s-Ginzburg-Ziv theorem, he studied zero-sum Ramsey numbers for \emph{non-abelian} $\Gamma$. More precisely, Caro showed that for any finite group $\Gamma$ and any $n$ such that $|\Gamma|$ divides $\binom{n}{2}$, there exists some $c(\Gamma)$ such that $R(K_n, \Gamma) \le n +c(|\Gamma|)$, where now we ask for a copy of $G$ and \emph{some} ordering of its edges such that the corresponding sum equals $0_\Gamma$. Until now, this is the only result we are aware of that establishes an upper bound for zero-sum Ramsey numbers over a large class of groups all at once.

Curiously, zero-sum Ramsey numbers are not monotone under graph containment (unlike their classical counterparts): for example, $R(C_4, \mathbb{Z}_2)=4$ but $R(2K_2, \mathbb{Z}_2)=5$, where $2K_2$ is a matching of size $2$. It therefore does not follow from Caro's result above that zero-sum Ramsey numbers of all graphs are small; and indeed, it is known that if $k=e(G)$ then $R(G, \mathbb{Z}_k)\geq R(G)$ \cite{BD92,Caro92}. This lower bound serves us as intuition for the correct behaviour of $R(G, \mathbb{Z}_k)$, and our main result provides an upper bound of the same order of magnitude for a large family of graphs $G$ and any finite abelian group $\Gamma$ with $|\Gamma|$ dividing $e(G)$. Namely, we prove that the zero-sum Ramsey number of any bounded-degree graph $G$ over any such $\Gamma$ is linear.

\begin{theorem}\label{main}
    Let $\Delta$ be a positive integer. Then there exists a constant $C=C(\Delta)$ such that for any graph $G$ with maximum degree $\Delta$ and any finite abelian group $\Gamma_0$ such that $|\Gamma_0|$ divides $e(G)$, we have

    $$ R(G, \Gamma_0)\leq C\cdot v(G).$$
\end{theorem}

Note that it suffices to prove \Cref{main} for $|\Gamma_0|=e(G)$, since in general we have $\Gamma_0\leqslant\Gamma_0 \times \mathbb{Z}_{e(G)/|\Gamma_0|}$ and the latter is a group of order precisely $e(G)$. It also clearly suffices to consider graphs without isolated vertices, in which case $v(G)$ and $e(G)$ are at most a factor of $\Delta$ apart. We henceforth assume that $|\Gamma_0|=e(G)=n$ and prove the above theorem with an upper bound of the form $C\cdot n$.

As mentioned above, we know that $R(G, \Gamma_0)\leq R_n(G)$, the $n$-color classical Ramsey number of $G$. This upper bound is in general much larger than linear in $n$. It is worth noting, however, that if we instead fix the number of colors to be $r$, say, then $R_r(G)$ is indeed linear in $n$.

\begin{theorem}[Chv\'atal, R\"odl, Szemer\'edi and Trotter \cite{ChRSzT}]
    \label{ramseylinear}
    Let $r, \Delta\geq 2$ be positive integers. Then there exists a constant $C'=C'(r, \Delta)$ such that for any graph $G$ with maximum degree $\Delta$, we have 
    \[R_r(G)\leq C'\cdot v(G).\]
\end{theorem}

The constant $C'$ obtained in \cite{ChRSzT} has a tower-type dependency on $\Delta$, owing to the use of the Szemer\'edi regularity lemma. For $r=2$, the best known upper bound of $C'(\Delta, 2)\leq 2^{c\Delta\log \Delta}$ for some absolute constant $c>0$ is due to Conlon, Fox and Sudakov \cite{CFS12} -- see also the survey \cite{CFS15} for a history of this problem; this is generally believed to not be tight, and the situation for $r\geq 3$ is even less clear. In our proof of \Cref{main}, however, we make use of \Cref{ramseylinear} (for general $r$), and hence the value of $C$ as a function of $\Delta$ depends on the value of $C'$ in \Cref{ramseylinear}. Due to this, we do not make an attempt to optimize our arguments for the value of $C$ throughout the paper, and our proof yields $C(\Delta)\leq \Delta^{42\Delta^6}\cdot C'(r,\Delta)$ for $r=200\Delta^{12}$. However, this constant could be sharpened in the setting of $v(G)\gg |\Gamma_0|$. We briefly discuss this in \Cref{sec.conclusion}, along with some interesting open questions.

\subsection{Terminology and notation}\label{sec.terminology}

Given an abelian group $\Gamma$, we abuse terminology by saying an element $g$ in $\Gamma$ has order $t\in \mathbb{Z}_{\geq 1}$ if $tg = 0$. We do not require $t$ to be the least such positive integer. Given a subset $S$ of a group $\Gamma$, we let $\langle S \rangle$ be the subgroup generated by $S$, that is, the smallest subgroup of $\Gamma$ containing $S$.    

Given $R$ a set of vertices, we denote $E(R)=\{\{x,y\}:x,y\in R')\}$. Informally, sets of vertices are treated as having a complete graph structure.

Given $G$ a graph, $\Gamma$ a finite abelian group, $R$ a set of vertices disjoint from $V(G)$, and $c:E(R)\to \Gamma$ a coloring of the edges of $R$ with elements from $\Gamma$, we establish the following notation. For $x,y,z\in R$, $u\in V(G)$, and $U,V\subset R$, we denote by $\overline{N}(u)$ the subgraph of $G$ induced by the vertices $N(u) \cup \{u\}$;  $xy$ the edge $\{x,y\}$;  $xyz\coloneqq \{xy,yz\}$  the path of length two from $x$ through $y$ to $z$; and $xV$ or $Vx$ the star $\{xv:v\in V\}$. Likewise, we define $yxV=\{yx\}\cup xV$, and $UyxV=Uy \cup yxV$. Note that $U$ and $V$ need not be disjoint. For any two sets $A, B$ containing elements of $\Gamma$, we write $A + B = \{a + b: a \in A, b \in B\}$ for the sumset of $A$ and $B$. In the case that $A =\{a\}$ contains a single element we let $a + B$ denote $\{a\} + B$. If $B$ is a subgroup of $\Gamma$ then $a + B$ is a \emph{coset} of $B$. For any graph $G'$ with vertex set in $R$, we set $c(G') \coloneqq  \sum_{e\in E(G')} c(e)$, e.g. $c(UyxV)=c(xy) +\sum_{u\in U} c(uy) +\sum_{v\in V} c(xv).$

We now fix the following notions for the remainder of this manuscript. 
Fix $G$ to be an arbitrary graph with maximum degree $\Delta$ and $e(G)=n$. Further, let $\Gamma_0$ be an arbitrary finite abelian group of order $n$. Then, let $R_0$ be a set of $C(\Delta)\cdot n$ vertices and fix $c_0:E(R_0)\to \Gamma_0$ an arbitrary edge coloring of $E(R_0)$.
To prove \Cref{main}, we wish to show that there is some injective mapping $f:V(G)\to R_0$ such that $$\sum_{xy\in E(G)}c_0(f(x)f(y))=0_{\Gamma_0}.$$

\subsection{Organization}
The rest of this paper is organized as follows. In \Cref{sec.example} we outline our general proof strategy. In \Cref{sec.algebra} we introduce our algebraic preliminaries. In \Cref{sec.blueprints,sec.gadgets} we introduce some of the combinatorial tools we will use and complete our setup. In \Cref{Sec.algorithm} we provide our algorithm and prove that the output is a zero-sum copy of $G$ (\Cref{main}). In \Cref{gadget_counts,Sec.vertexcounts,sec.nofail} we verify the correctness of the algorithm. We conclude with a discussion of open problems and further work in \Cref{sec.conclusion}, and prove a certain technical lemma in the Appendix.

\section{Intuition and proof outline for a path}\label{sec.example}

\begin{figure}[htb]
\captionsetup{justification=centering}
\centering
        \begin{tikzpicture}

        \fill[] (-1, 0) circle(0.05) node[above] {$u$};
        \fill[] (1, 1) circle(-0.05) node[above] {$w_1$};
        \fill[] (1, -1) circle(-.05) node[below] {$w_2$};
        \fill[] (3, 0) circle(0.05) node[above] {$v$};
        
        \draw[thick] (-1, 0) -- (1, 1) -- (3,0);

        \draw[thick] (-1, 0) -- (1, -1) -- (3,0);

        \end{tikzpicture}
\caption{Example of a $(2,2)$ gadget.
}\label{simple_gadget}
\end{figure}
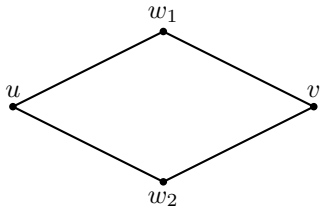

Suppose for now, that $G$ is a path of length $p$ and $\Gamma_0=\mathbb{Z}_p$, where $p$ is prime. Then consider \Cref{simple_gadget} and think of it as two paths of length two going from $u$ to $v$. Now suppose that we can find four such vertices in $R_0$ so that these two paths sum to different values: $c_0(uw_1v)\neq c_0(uw_2v)$. We will refer to such a copy as a gadget. Now, suppose that we manage to inductively find $M$ such gadgets in $R_0$, say $\{u_1,w_{1,1},w_{1,2},v_1\},\ldots,\{u_M,w_{M,1},w_{M,2},v_M\}$, for some large $M$, and that moreover these are pairwise disjoint except for possibly having $v_i=u_{i+1}$ for some values of $i$. Write $A_i=\{c_0(u_iw_{i,1}v_i), c_0(u_iw_{i,2}v_i)\}$, and note this is a subset of $\Gamma_0$ of cardinality precisely two. Now chain the gadgets together, meaning that for every $i\leq M-1$ we either have $v_i=u_{i+1}$ or we add the edge $v_iu_{i+1}$, and consider the paths from $u_1$ to $v_M$ obtained by walking along the vertices of each gadget in turn. Note that the set of possible sums obtained along such paths from $u_1$ to $v_M$ is exactly $A_1+c_0(v_1u_2)+A_2+ c_0(v_2u_3)+\dots+c_0(v_{M-1}u_M)+A_M$, where we set $c_0(v_iu_{i+1})=0_\Gamma$ if $v_i=u_{i+1}$. Now, by the Cauchy-Davenport theorem (\Cref{Kneser} applied to cyclic groups of prime order), this set has cardinality at least $\min\{M+1, p\}$. Then, for $M+1\geq p$, we can obtain every element of $\Gamma_0$ as the sum of such a path, and in particular there exists a zero-sum path.

If instead we can't find $M=p-1$ such gadgets in $R_0$, then we know that there exists a vertex subset of $R_0$ of order at least $|R_0|-4p$ on which the coloring is `well-behaved': any two paths of length two between two fixed vertices $x,y$ have the same value. A simple analysis now shows that such a coloring must be monochromatic, and so we find a monochromatic (and hence zero-sum) copy of $G$.

In the case of general $G$ and $\Gamma_0$ as in \Cref{main}, there are several complications that arise when trying to follow the above approach. We go through them in order and give an indication of how our proof will handle them.

\begin{enumerate}
    \item Although we can certainly guarantee to find $p-1$ gadgets (or otherwise find a sufficiently large complete subgraph on which $c_0$ is monochromatic), if we chain these gadgets together, the length of the resulting paths will be more than $p$ (in fact, it could be as large as $3p$).
    
    Note, however, that if we could find paths of a fixed length $\ell\leq p$ between some fixed $u_1$ and $v_M$ which sum up to every element of $\Gamma_0$, then we can simply extend these by appending a fixed path of length $p-\ell$ to $v_M$, and this now yields a zero-sum path of length exactly $p$. In order to do this, we will increase the `multiplicity' of the gadgets we seek: instead of $w_1$ and $w_2$, we will find $w_1$, $w_2$,  $\ldots$, $w_\beta$ for some carefully chosen $\beta\geq 2$, with the property that $c_0(uw_iv)$ are pairwise distinct for $1\leq i\leq \beta$. Then we will need only about $ n/(\beta-1)$ gadgets to obtain paths of some fixed length of at most $3n/(\beta-1)$ which generate all of $\Gamma_0$. \label{multiplicity}
    
    \item Unlike a path, a general bounded degree graph may not contain any vertices of degree two. The above gadgets require that the \emph{free vertices} $w_1$, $w_2$ correspond to a vertex of degree precisely two in $G$ (but this is not required for the \emph{fixed} $u$'s and $v$'s).

     Given a vertex of degree $d$ in our graph $G$, we can replace $u$ and $v$ in \Cref{simple_gadget} with some set $U$ of $d$ vertices (these will be called the \emph{fixed vertices}) and then we will have a set $\{w_1,\ldots,w_\beta\}$ that represents our free vertex of degree $d$. The issue is slightly more subtle, however. The absence of such a gadget does not tell us something strong enough about the coloring being well-behaved, so instead we need to take two adjacent vertices in $G$ with degrees $d$ and $d'$ and we find a gadget consisting of sets $U$ of size $d-1$, $V$ of size $d'-1$ and some free vertices, so that the set of $Uww'V$, where $w$ and $w'$ are free, give us at least $\beta$ different values. A precise definition of a gadget will follow in \Cref{sec.gadgets}; see \Cref{dd'_gadget}. \label{degree}
     
    \item The Cauchy-Davenport theorem only applies in $\mathbb{Z}_p$. We instead use a result due to Kneser (\Cref{Kneser}) which is similar to Cauchy-Davenport and applies to general finite abelian groups, but it only guarantees that our sumset contains some coset of a subgroup of $\Gamma_0$. Because of this, we will find gadgets one by one, and at each step keep track of the sumset they generate. As soon as we generate some coset $r+H\subset \Gamma_0$, we quotient our coloring by $H$ and proceed to find more gadgets in this new coloring. Doing this carefully until the end, we will have generated all of $\Gamma_0$ via our set of gadgets.\label{CD_problem}
    
    \item If $\Gamma_0$ is an arbitrary finite abelian group, then a well-behaved coloring of $R_0$ as above no longer needs to be monochromatic. If the free vertex of $G$ that we are considering has degree $d$ (the above exposition was for $d=2$, and we have hinted at our approach for general $d$ in \cref{degree} above), and if $\Gamma_0$ has elements of order $d$, then we can no longer deduce that $c_0$ is monochromatic on $R_0$, but instead we can ensure that there exists some subgroup $H\subset \Gamma_0$ so that the quotient coluring $c_0/H$ is highly structured in a certain sense, which we dub as $c_0$ being \emph{$d$--well-behaved}. As it turns out, the correct notion of well-behaved for us will be the existence of some \emph{vertex}-coloring such that the color of any edge $xy$ is controlled by the color of its two endpoint vertices. This is defined formally in \Cref{Sec.algorithm}. 

    Our coloring being $d$--well-behaved on its own will not be enough to find a zero-sum embedding of $G$, but, if we look for a minimal such case (with respect to the size of the subgroup $H\leqslant \Gamma_0$) from the start, then we can ensure that we will always be able to find our desired gadgets. We will then proceed to algorithmically find gadgets one by one, and at every step consider the sumset (in $\Gamma)$ given by our current gadgets. Quotienting out our coloring $c_0$ every time we obtain a larger subgroup, we then carefully find the next gadget. \label{WB_problem}
\end{enumerate}

\section{Algebra basics}\label{sec.algebra}

We now take stock of some of the elementary algebraic tools we will require. First and foremost, we will be making extensive use of the classification theorem of finite abelian groups.

\begin{theorem}[Fundamental theorem of finite abelian groups\cite{dummit2004}]\label{thm:bascigroupthm}
 Let $\Gamma$ be an abelian group of order $n > 1$ and let the unique factorization of $n$ into distinct prime powers be
 $n = p_1^{\alpha_1} p_2^{\alpha_2} \cdots p_k^{\alpha_k}.$
 Then
\begin{enumerate}
    \item $\Gamma \cong A_1 \times A_2 \times \cdots \times A_k$ for groups $A_1, \dots, A_k$ with $|A_i| = p_i^{\alpha_i}$ for each $1\leq i\leq k$, and 
    
    \item for each $A \in \{A_1, A_2, \ldots, A_k\}$ with $|A| = p^{\alpha}$, we have
    \[
    A \cong \mathbb{Z}_{p^{\beta_1}} \times \mathbb{Z}_{p^{\beta_2}} \times \cdots \times \mathbb{Z}_{p^{\beta_\ell}}
    \]
    for some $\beta_1 \geq \beta_2 \geq \cdots \geq \beta_\ell \geq 1$ satisfying
    $\beta_1 + \beta_2 + \cdots + \beta_\ell = \alpha \text{ (where $\ell$ and $\beta_1, \ldots, \beta_\ell$ depend on $i$)}$. 
\end{enumerate}

Moreover, as $\ZZ_{mn} \cong \ZZ_m \times \ZZ_n$ if and only if $m$ and $n$ are coprime, it follows that each finite abelian group also has a unique factorization $\Gamma \cong \ZZ_{m_1} \times  \cdots \times \ZZ_{m_{k'}}$ where $m_i | m_{i + 1}$ for each $1 \le i < k'$. This decomposition is called an \emph{invariant factor decomposition}.
\end{theorem}

The other algebraic result we will need is the following generalization of the Cauchy-Davenport theorem.

\begin{theorem}[Kneser's Theorem \cite{Kneser_1953}]\label{Kneser}
Let $\Gamma$ be a finite abelian group and $A,B\se \Gamma$. Then $|A+B|\geq \min\{|\Gamma|, |A|+|B|-1\}$ or $A+B$ contains a coset of some nontrivial subgroup of $\Gamma$.
\end{theorem}


    


\section{Finding blueprints}\label{sec.blueprints}

In this section, we break up the graph $G$ into `blueprint pairs', which we intuitively think of subgraphs of $G$ that will be suitable candidates for what our future gadgets should look like. In the simple case of a path we sketched in \Cref{sec.example}, the blueprint pairs were subpaths of length $2$, and they gave rise to $(2, 2)$--gadgets. We formalize this idea in \Cref{blueprints} and obtain a suitable family of blueprint pairs from the graph $G$ in \Cref{K_partition}.

\begin{defn} \label{blueprints}
    Given a graph $G$, a \emph{blueprint} of type $d$ is an induced subgraph $\Upsilon$ of $G$ of the form $\overline{N}(v)$ for some vertex $v$ of degree $d$, equipped with an identification of the \emph{free vertex} $\iota(\Upsilon)=v$ and the \emph{fixed vertices} $\zeta(\Upsilon)\coloneqq N(\iota(\Upsilon))$. We informally write $\overline{N}(v)$ for a blueprint, always meaning that the free vertex is $v$.
    
    Given two blueprints $\Upsilon_1$ and $\Upsilon_2$ of type $d_1$ and $d_2$ with $m\coloneqq |N(\iota(\Upsilon_1))\cap N(\iota(\Upsilon_2))|$, we define an \emph{blueprint pair} $(\Upsilon_1,\Upsilon_2)$ of type $(d_1,d_2,m)$ on the condition that $\iota(\Upsilon_1)$ and $\iota(\Upsilon_2)$ are adjacent in $G$. The definitions of $\iota$ and $\zeta$ are naturally extended for blueprint pairs to be $\iota(\Upsilon_1,\Upsilon_2) = \{\iota(\Upsilon_1),\iota(\Upsilon_2)\}$ and $\zeta(\Upsilon_1,\Upsilon_2) =(\zeta(\Upsilon_1)\cup \zeta(\Upsilon_2))\setminus \iota(\Upsilon_1,\Upsilon_2)$. Let $V(\Upsilon_1,\Upsilon_2)=V(\Upsilon_1)\cup V(\Upsilon_2)$ be the set of all vertices in the blueprint pair.

    Finally, two blueprint pairs $(\Upsilon_1,\Upsilon_2)$ and $(\Upsilon'_1,\Upsilon'_2)$ are \emph{non-overlapping} if their vertex sets are disjoint.
\end{defn}

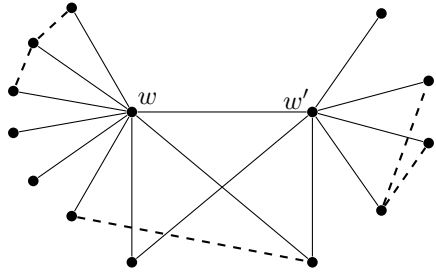
\begin{figure}[htb]
\captionsetup{justification=centering}
\centering
\begin{tikzpicture}[
  v2/.style={fill=black,minimum size=4pt,ellipse,inner sep=1pt},
  node distance=1.5cm,scale=0.8
]

\node[v2] (w1) at (2,0) {};
\node[v2] (w2) at (-1,0) {};

\node[v2] (u1) at ({-1-2*cos(60)}, { 2*sin(60)}) {};
\node[v2] (u2) at ({-1-2*cos(35)}, { 2*sin(35)}) {};
\node[v2] (u3) at ({-1-2*cos(10)}, { 2*sin(10)}) {};
\node[v2] (u4) at ({-1-2*cos(10)}, {-2*sin(10)}) {};
\node[v2] (u5) at ({-1-2*cos(35)}, {-2*sin(35)}) {};
\node[v2] (u6) at ({-1-2*cos(60)}, {-2*sin(60)}) {};

\node[v2] (v1) at ({2+2*cos(55)}, { 2*sin(55)}) {};
\node[v2] (v2) at ({2+2*cos(15)}, { 2*sin(15)}) {};
\node[v2] (v3) at ({2+2*cos(-15)}, {2*sin(-15)}) {};
\node[v2] (v4) at ({2+2*cos(-55)}, {2*sin(-55)}) {};

\node[v2] (z1) at (-1,-2.5) {};
\node[v2] (z2) at (2,-2.5) {};

\draw (w2) -- (u1);
\draw (w2) -- (u2);
\draw (w2) -- (u3);
\draw (w2) -- (u4);
\draw (w2) -- (u5);
\draw (w2) -- (u6);
\draw[dashed,thick] (u1) -- (u2);
\draw[dashed,thick] (u2) -- (u3);
 
 \draw (w1) -- (v1);
\draw (w1) -- (v2);
\draw (w1) -- (v3);
\draw (w1) -- (v4); 

\draw[dashed,thick] (v4) -- (v2);
\draw[dashed,thick] (v4) -- (v3);

 \draw (w1) -- (z1);
\draw (w1) -- (z2);
 \draw (w1) -- (w2);
 \draw (w2) -- (z1);
\draw (w2) -- (z2);

\draw[dashed,thick] (u6) -- (z2);
\node[xshift=-6pt,yshift=5pt] at (w1) {$w'$};
\node[xshift=6pt,yshift=5pt] at (w2) {$w$};

\end{tikzpicture}
   
\caption{An example of a $(9,7,2)$ blueprint pair, an induced subgraph of $G$. The dashed lines represent edges that are present in $G$, but do not affect the behavior of realizations of gadgets (see \Cref{sec.gadgets}).}\label{blueprint_example}
\end{figure}

The first step towards proving \Cref{main} is to extract an appropriate set of pairwise non-overlapping blueprint pairs from $G$ which we will eventually realize as gadgets in $R_0$ (the precise definition of a gadget will be given later).
First, consider an auxiliary coloring $\hat{c}$ of $E(G)$ as follows. To each edge $xy\in E(G)$ we assign the color $\hat{c}(xy)=\gcd (d(x),d(y))$, where $d(x)$ is the degree of $x$ in $G$. Let $\kappa'$ be the most common color under $\hat{c}$ (or one such color chosen arbitrarily, if there are multiple), set
\beq{def:kappa}\kappa \coloneqq  \gcd(n,\kappa') \enq
and let $K'$ be the set of blueprint pairs defined by $$K'=\{(\overline{N}(x),\overline{N}(y)):\hat{c}(xy)=\kappa'\}$$
and note that $|K'|\geq n/\Delta$.

Note that if $g\in \Gamma_0$ satisfies $\kappa'g = 0$, then $\kappa g= 0$.

Now let $K''$ be a maximal pairwise non-overlapping subset of $K'$. For every edge $xy$, there exist at most $2\Delta^3$ edges $x'y'\in E(G)$ such that $(\overline{N}(x),\overline{N}(y))$ and $(\overline{N}(x'), \overline{N}(y'))$ overlap, and thus $|K''|\geq n/(2\Delta^4)$.

We now partition the vertices of $G$ into parts so that the sum of degrees of each part is divisible by $\kappa$. Note that $\sum_{v\in G}d(v) = 2|E(G)|=2n \equiv 0 \pmod \kappa$, by definition of $\kappa$, so there is no divisibility obstruction to this. We start by partitioning $V(G)$ into sets $\Pi_i=\{v\in V(G):d(v)=i\}$ for all $1\leq i\leq \Delta$. Now set $K=\emptyset$. 

Then we do the following until $K''$ is empty:
\begin{enumerate}\label{def:K}
    \item Take some 
$(\Upsilon,\Upsilon')$ from $K''$ and set $K''\gets K''\setminus \{(\Upsilon,\Upsilon')\}$. \label{K_1first}
\item Consider the set of fixed vertices $\zeta=\zeta(\Upsilon,\Upsilon')$, and order its elements in some arbitrary way as $\zeta=\{v_1, \dots, v_{|\zeta|} \}$. For each $1\leq i\leq |\zeta|$ in turn define a vertex set $\pi_{v_i}$ as follows. If $d(v_i) \equiv  0 \pmod \kappa$, then set $\pi_{v_i}=\{v_i\}$; otherwise add $v_i$ together with $\kappa-1$ arbitrary vertices from $\Pi_{d(v_i)}\setminus \left(\zeta\cup \bigcup_{j=1}^{i-1}\pi_{v_j} \right)$ to $\pi_{v_i}$. Notice that in either case $\pi_{v_i}$ does not contain any of the free vertices of the blueprints in $K''$, since by construction all free vertices have degrees divisible by $\kappa$.
\item If for some $1\leq i\leq |\zeta|$ we have $d(v_i)\not\equiv 0 \pmod \kappa$ and $\left|\Pi_{d(v_i)}\setminus\left(\zeta\cup\bigcup_{j=1}^{i-1}\pi_{v_j}\right)\right|<\kappa-1$, then return to step 1. \label{K_13}
\item For every $1\leq i\leq |\zeta|$ in turn, update $\Pi_{d(v_i)}\gets \Pi_{d(v_i)}\setminus\pi_{v_i}$.
\item If any of the added vertices belongs to any other blueprint pairs in $K''$, then remove all such pairs from $K''$. In other words, update $K''\gets K''\setminus \{(\Upsilon_1,\Upsilon_2)\in K'':V(\Upsilon_1,\Upsilon_2)\cap \left(\bigcup_{v\in \zeta}\pi_v\right)\neq \emptyset\}$. \label{K_15}
\item Update $K\gets K\cup \{(\Upsilon,\Upsilon')\}$ and return to step 1.
\end{enumerate}

We now count how many $(\Upsilon,\Upsilon') \in K''$ are discarded in steps 3 and 5 of the above process. In step 3, for each $d\in [\Delta]$ we remove from $K''$ at most $2\Delta^2$ blueprint pairs because of some $v\in\zeta$ with $d(v)=d$. Indeed, if we encounter such a $v$, say $v_i$, then at that iteration (and all subsequent ones) we must have $\left|\Pi_{d}\right|<\kappa|\zeta|\leq 2\Delta^2$. Hence for a given $d \in [\Delta]$ there are at most $2\Delta^2$ choices for such a $v$ and each of them belongs to at most one $\zeta$. This gives at most $2\Delta^3$ blueprint pairs discarded because of step 3.

In step 5, for each blueprint pair added to $K$, we remove from $K''$ at most $|\zeta|(\kappa-1)\leq(d_1+d_2-2)(\kappa-1) < 2\Delta^2$ pairs. Thus, we finish the above process with 

$$|K|\geq \frac{n/(2\Delta^4) - 2\Delta^3}{2\Delta^2} = \frac{n-2\Delta^7}{4\Delta^6}\ge \frac{n}{5\Delta^6}.$$
Throughout the operation of the algorithm in \Cref{Sec.algorithm}, $K$ will be our pool of available blueprint pairs.

With the set $K$ now in hand, define
$$P' \coloneqq  \{\pi_v: v\in \zeta(\Upsilon,\Upsilon') \text{ for some }  (\Upsilon,\Upsilon') \in K\}$$
which will form part of our vertex partition of $G$, and set
\beq{G'def}G'\coloneqq  G\setminus \bigcup_{(\Upsilon,\Upsilon')\in K}\iota(\Upsilon,\Upsilon')\enq
and likewise, for each $d\in [\Delta]$ update $$\Pi_d\gets \Pi_d \setminus \bigcup_{(\Upsilon,\Upsilon')\in K}\iota(\Upsilon,\Upsilon')$$ to remove all the free vertices from each of the $\Pi_d$.

Next, enumerate each $\Pi_d=\{v_{(d,1)},\ldots,v_{(d,l_d)}\}$ and let  $l_d'=\lfloor l_d/\kappa \rfloor$. Now, for each $j\in [l'_d]$ we define $$\pi_{(d,j)}\coloneqq \{v_{(d,(j-1)\kappa +1)},\ldots,v_{(d,j\kappa)}\}.$$ The sets $\pi_{(d, j)}$ will also form part of our vertex partition of $G$.

Finally, consider the set of leftover vertices $\pi'\coloneqq \bigcup_{d\in [\Delta]} \{v_{(d,l'_d\kappa +1)},\ldots,v_{(d,l_d)}\}$. Then let 
\begin{equation}\label{equ:P}
    P\coloneqq \{\pi'\} \cup \{\pi_{(d,j)}:d\in [\Delta], j\in [l'_d]\}.
\end{equation} 
Now, for each set $X\in P$ the sum of degrees $\sum_{v\in X} d(v)$ is divisible by $\kappa$. For $X\neq \pi'$ this is trivial, and for $X=\pi'$ we recall that $P$ is a vertex partition of $G'\setminus\bigcup P'$ and $\sum_{v\in G' \setminus \bigcup P'}d(v) \equiv 0 \pmod \kappa$, so $\sum_{v\in \pi'}d(v)\equiv 0 \pmod \kappa$. Also note that we have $|X|\leq \kappa$ for all $X\in P$ except when $X=\pi'$, in which case $|X|\leq (\kappa-1)\Delta<\Delta^2$.

Then, finally we define our partition of $G'$ to be 
\begin{equation}\label{eq.G'partition}
    \Pi \coloneqq  P' \cup P.
\end{equation}
Now, recall that all the vertices in $G\setminus G'$ are free and therefore have degree divisible by $\kappa$, and so $\Pi$ together with the singletons of all free vertices in $K$ forms the desired vertex partition of $G$.

Recalling the discussion in \cref{multiplicity} in \Cref{sec.example}, we now attempt to briefly motivate our choice of multiplicity at which we will find gadgets, as this is closely connected to the cardinality of the set $K$ of blueprint pairs we obtained in this section. Naively, if each blueprint pair has multiplicity $\beta$ (so that it contributes $\beta$ colors to the count in our adaptation of the Cauchy-Davenport argument), then we require $|K|\cdot (\beta-1) \geq n$, and solving this yields $\beta = 5\Delta^6+1$. However, it turns out that we can't find gadgets at multiplicity $\beta$ throughout our entire (algorithmic) proof, and that some constant fraction of them will need to be at multiplicity 2. As such, for the remainder of the paper we define
\begin{equation}\label{equ:alphabeta}
    \alpha\coloneqq  10\Delta^6 \text{ and } \beta\coloneqq  2\alpha=20\Delta^6.
\end{equation} 
The parameter $\beta$ is the multiplicity of the gadgets we will be looking for in the first few rounds of the algorithm, and $\alpha$ is a measure of how long into the process we use gadgets of multiplicity $\beta$ before switching to finding gadgets of multiplicity 2: we will look for $\beta$-multiplicity gadgets until we have reduced our group $\Gamma_0$ to some $\Gamma_0/H$ of order $|\Gamma_0|/\alpha$, where $H$ is a subgroup of $\Gamma_0$.

 The above process implies the following:
\begin{proposition}\label{K_partition}
     Let $\kappa'$ be the most common element of the multiset $\{\gcd (d(x),d(y)): xy \in E(G)\}$ and let $\kappa = \gcd( \kappa', e(G))$. There exists a collection $K$ of non-overlapping blueprint pairs as defined in \Cref{blueprints} along with an associated partition $\Pi_0$ of $V(G)$ such that the following all hold:
     \begin{enumerate}[label=\upshape{(\roman{enumi})}]
     \item $|K| \ge \frac{e(G)}{5\Delta ^6}$;
     \item every blueprint pair in $K$ has type $(d_1, d_2, m)$ for some $d_1, d_2\leq \Delta$ which satisfy $\gcd(d_1, d_2, e(G))=\kappa$;
     \item for each vertex $v \in V(G)$, there is some $\pi_v\in \Pi_0$ such that $v\in \pi_v$ and $\sum_{u\in \pi_v}d(u)\equiv 0\pmod \kappa$;
     \item if $u$ and $v$ are vertices of some blueprint pairs in $K$, then we have $\pi_u\cap \pi_v = \emptyset$; and
     \item for each $\pi\in \Pi_0$, $|\pi|<\Delta^2$.
     \end{enumerate}
 \end{proposition}

\section{Gadgets and realizations}\label{sec.gadgets}

\Cref{K_partition} provides us with our pool of blueprint pairs and allocates to each of the fixed vertices $v$ a set of vertices $\pi_v$ (which are pairwise disjoint) such that $\sum_{u\in \pi_v} d(u) \equiv0\pmod \kappa$. Equipped with this, we are finally in a position to formally define our notions of a gadget and a realization of a blueprint.  

\begin{defn} \label{def:gadgets}
    Given a finite abelian group $\Gamma$, a vertex set $R$, a graph $\gamma$ with vertex set in $R$, an edge coloring $c: E(R) \to \Gamma$ and a vertex coloring $\cC:R \to \Gamma$, we say $\gamma$ is a $(d_1,d_2,m,\lambda)$--\emph{gadget} if $V(\gamma)$ can be partitioned into subsets $V(\gamma)=D_1 \sqcup D_2 \sqcup M \sqcup X_1\ \sqcup X_2 \sqcup \bigsqcup_{v\in D_1 \cup D_2 \cup M} P_v$ such that $|D_1|=d_1-m-1,$ $|D_2|=d_2-m-1,$ $|M|=m$, and $|X_1|,|X_2|\leq \lambda$ so that
    
    \begin{enumerate}
        \item the subgraph of $\gamma$ induced by each of the pairs $(D_1,X_1),$ $(X_1,X_2),$ $(X_2,D_2)$, $(M,X_1)$, $(M,X_2)$ is a complete bipartite graph;
        \item for all $v\in D_1\cup D_2\cup M$ and $w\in P_v$, we have $\cC(w)=\cC(v)$; and
        \item defining $D_1'\coloneqq D_1\cup M$ and $D_2'\coloneqq  D_2\cup M$, we  have
        \begin{equation}\label{colorful}|\{c(D_1'v_1v_2D_2') +  d_1\cdot \cC(v_1) + d_2\cdot \cC(v_2):v_1\in X_1,v_2\in X_2\}|\geq \lambda.
    \end{equation}
    \end{enumerate}
\end{defn}

\begin{figure}[htb]
\captionsetup{justification=centering}
\centering
        \begin{tikzpicture}[scale=1.3]

        \draw[] (-2.8,0) circle(0.5) node {$D_1$};
        \fill[] (-2.5, -.2) circle(0.05) node[above]{};
        \fill[] (-2.8, 0.4) circle(0.05) node[above]{};
        \fill[] (-2.5, .2) circle(0.05) node[above]{};
        \fill[] (-2.8, -.4) circle(0.05) node[above]{};
        \fill[] (-3.1, .2) circle(0.05) node[above]{};
        \fill[] (-3.1, -.2) circle(0.05) node[above]{};
 
        \node[] at (-1,-.3) {$X_1$};
        \draw[] (-1,0) circle(0.75);
        \fill[] (-1, .2) circle(0.05) node[above] {$w_1$};
        \fill[] (-1.4, 0) circle(0.05) node[above] {$w_2$};
        \fill[] (-.6, 0) circle(0.05) node[above] {$w_3$};

        \node[] at (1,-.3) {$X_2$};
        \draw[] (1,0) circle(0.75);
        \fill[] (1, .2) circle(0.05) node[above] {$w'_1$};
        \fill[] (1.4, 0) circle(0.05) node[above] {$w'_2$};
        \fill[] (.6, 0) circle(0.05) node[above] {$w'_3$};

        \draw[] (2.8,0) circle(0.5) node {$D_2$};
        \fill[] (2.5, -.2) circle(0.05) node[above]{};
        \fill[] (2.5, .2) circle(0.05) node[above]{};
        \fill[] (3.1, .2) circle(0.05) node[above]{};
        \fill[] (3.1, -.2) circle(0.05) node[above]{};

        \fill[] (-.25, -1.3) circle(0.05) node[above]{};
        \fill[] (.25, -1.3) circle(0.05) node[above]{};
        \draw[] (0,-1.3) circle(0.4) node {$M$};

        \draw[thick] (2.64, 0.47) -- (1.75, 0);
        \draw[thick] (2.64, -0.47) -- (1.75, 0);

        \draw[thick] (-2.64, 0.47) -- (-1.75, 0);
        \draw[thick] (-2.64, -0.47) -- (-1.75, 0);

        \draw[thick] (0, -.9) -- (-1, -.75);
        \draw[thick] (-.4, -1.3) -- (-1, -.75);

        \draw[thick] (0, -.9) -- (1, -.75);
        \draw[thick] (.4, -1.3) -- (1, -.75);

        \draw[dashed] (.25, 0) -- (-.25, 0);
        \draw[dashed] (.35, .4) -- (-.35, .4);
        \draw[dashed] (.35, -.4) -- (-.35, -.4);
        \draw[dashed] (.35, .4) -- (-.35, -.4);
        \draw[dashed] (.35, -.4) -- (-.35, .4);
        
        \end{tikzpicture}
\caption{Example of a $(9,7,2,3)$ gadget (the sets $P_v$ are not shown). In comparison to Figure \ref{simple_gadget}, instead of paths from $v$ to $u$ we have $6$ edges from $D_1$ to each $w_i$, $4$ edges from $D_2$ to each $w_j'$, 2 edges from $M$ to each $w_i$ and $w_j'$, and every edge $w_iw_j'$. This gadget corresponds to the blueprint in \Cref{blueprint_example}.}
\label{dd'_gadget}
\end{figure}
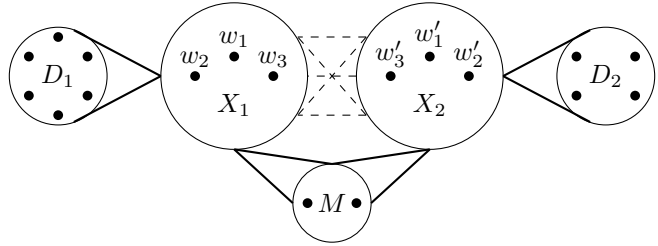

If $|X_2|=1$, then note that the choice of $D_2$ does not affect \cref{colorful}. Further, since the sole purpose of $M$ is to allow the neighborhoods of $X_1$ and $X_2$ to intersect (this will correspond to triangles containing both of the free vertices in the blueprint pair), we can set $D_2=\emptyset$ and refer to the resulting gadget as a $(d_1,\lambda)$--gadget, and regard the relevant vertex sets as $D_1\cup M\cup X_1$ and $\emptyset$ instead of $D_1$ and $M$, respectively (see Figure \ref{fig:(d,l)-gadget}).

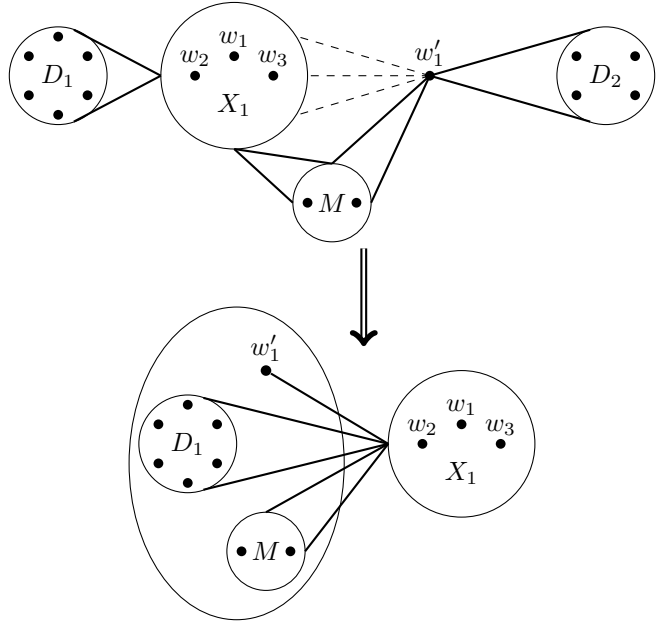
\begin{figure}[htb]
\captionsetup{justification=centering}
\centering
        \begin{tikzpicture}[scale=1.3]

        \draw[] (-2.8,0) circle(0.5) node {$D_1$};
        \fill[] (-2.5, -.2) circle(0.05) node[above]{};
        \fill[] (-2.8, 0.4) circle(0.05) node[above]{};
        \fill[] (-2.5, .2) circle(0.05) node[above]{};
        \fill[] (-2.8, -.4) circle(0.05) node[above]{};
        \fill[] (-3.1, .2) circle(0.05) node[above]{};
        \fill[] (-3.1, -.2) circle(0.05) node[above]{};
 
        \node[] at (-1,-.3) {$X_1$};
        \draw[] (-1,0) circle(0.75);
        \fill[] (-1, .2) circle(0.05) node[above] {$w_1$};
        \fill[] (-1.4, 0) circle(0.05) node[above] {$w_2$};
        \fill[] (-.6, 0) circle(0.05) node[above] {$w_3$};

        \fill[] (1, 0) circle(0.05) node[above] {$w'_1$};

        \draw[] (2.8,0) circle(0.5) node {$D_2$};
        \fill[] (2.5, -.2) circle(0.05) node[above]{};
        \fill[] (2.5, .2) circle(0.05) node[above]{};
        \fill[] (3.1, .2) circle(0.05) node[above]{};
        \fill[] (3.1, -.2) circle(0.05) node[above]{};

        \fill[] (-.25, -1.3) circle(0.05) node[above]{};
        \fill[] (.25, -1.3) circle(0.05) node[above]{};
        \draw[] (0,-1.3) circle(0.4) node {$M$};

        \draw[thick] (2.64, 0.47) -- (1, 0);
        \draw[thick] (2.64, -0.47) -- (1, 0);

        \draw[thick] (-2.64, 0.47) -- (-1.75, 0);
        \draw[thick] (-2.64, -0.47) -- (-1.75, 0);

        \draw[thick] (0, -.9) -- (-1, -.75);
        \draw[thick] (-.4, -1.3) -- (-1, -.75);

        \draw[thick] (0, -.9) -- (1, 0);
        \draw[thick] (.4, -1.3) -- (1,0);

        \draw[dashed] (1, 0) -- (-.25, 0);
        \draw[dashed] (1, 0) -- (-.35, .4);
        \draw[dashed] (1, 0) -- (-.35, -.4);

        \end{tikzpicture}

        \begin{tikzpicture}[
        v2/.style={fill=black,minimum size=4pt,ellipse,inner sep=1pt},
        node distance=1.5cm,scale=1.3]

        \draw[->,thick, double distance = 1] (-1,2) -- (-1,1);
        \draw[] (-2.8,0) circle(0.5) node {$D_1$};
        \fill[] (-2.5, -.2) circle(0.05) node[above]{};
        \fill[] (-2.8, 0.4) circle(0.05) node[above]{};
        \fill[] (-2.5, .2) circle(0.05) node[above]{};
        \fill[] (-2.8, -.4) circle(0.05) node[above]{};
        \fill[] (-3.1, .2) circle(0.05) node[above]{};
        \fill[] (-3.1, -.2) circle(0.05) node[above]{};
 
        \node[] at (0,-.3) {$X_1$};
        \draw[] (0,0) circle(0.75);
        \fill[] (0, .2) circle(0.05) node[above] {$w_1$};
        \fill[] (-.4, 0) circle(0.05) node[above] {$w_2$};
        \fill[] (.4, 0) circle(0.05) node[above] {$w_3$};

        \node[v2] (w1) at (-2,.75) {};
        \node[above] at (w1)  {$w'_1$};

        \fill[] (-2.25, -1.1) circle(0.05) node[above]{};
        \fill[] (-1.75, -1.1) circle(0.05) node[above]{};
        \draw[] (-2,-1.1) circle(0.4) node {$M$};

        \draw (-2.3,-.20) ellipse [x radius=1.1, y radius=1.6];

        \draw[thick] (w1) -- (-.75, 0);

        \draw[thick] (-2.64, 0.47) -- (-.75, 0);
        \draw[thick] (-2.64, -0.47) -- (-.75, 0);

        \draw[thick] (-2, -.7) -- (-.75, 0);
        \draw[thick] (-1.6, -1.1) -- (-.75,0);

        \end{tikzpicture}
\caption{Example of a $(9,7,2,3)$--gadget (the sets $P_v$ are not shown) with $|X_2|=1$ and its reformulation as a $(9,3)$--gadget.}
\label{fig:(d,l)-gadget}
\end{figure}
    In what follows, we will only be interested in $(d_1,\lambda)$-- and $(d_1,d_2,m,\lambda)$--gadgets where $\gcd\{d_1,d_2\}=\kappa'$, with $\kappa'$ as defined in \cref{def:kappa}.

\begin{defn}\label{def:realization}
    Given a vertex set $R$, an edge coloring $c$ of $E(R)$, and a vertex coloring $\cC$ of $R$, we say that a blueprint pair $(\Upsilon_1,\Upsilon_2)$ of type $(d_1,d_2,m)$ is \emph{realized} (or found) at \emph{multiplicity} $\lambda$ if we can find a $(d_1,\lambda)$--, $(d_2,\lambda)$--, or $(d_1,d_2,m,\lambda)$--gadget $\gamma$ in $R$ with $V(\gamma)=D_1\cup X_1 \cup X_2 \cup D_2 \cup M \cup \left(\bigcup_{v\in D_1\cup D_2\cup M} P_v\right)$, and a family $F$ of functions from $V_F=V(\Upsilon_1,\Upsilon_2)\cup \left(\bigcup_{u\in\zeta(\Upsilon_1,\Upsilon_2)}\pi_{u}\right)$ to $V(\gamma)$, such that the following all hold: 
    
    \begin{enumerate}
        \item (fixed bundles are the correct size) for all $u\in \zeta(\Upsilon_1,\Upsilon_2)$ and $f\in F$, we have $|P_{f(u)}|=|\pi_u|-1$;\label{Bundles_correct}
        \item (fixed vertices are constant and distinct) for all $f,f'\in F$, if $u,u'\in  \zeta(\Upsilon_1, \Upsilon_2)$, then $f(u)=f'(u')$ if and only if $u=u'$;
        \item (fixed vertices map to the correct set) for all $f\in F$, if $u\in \zeta(\Upsilon_1, \Upsilon_2)\setminus \zeta(\Upsilon_2)$ then $f(u)\in D_1$; if $u\in \zeta(\Upsilon_1, \Upsilon_2)\setminus \zeta(\Upsilon_1)$ then $f(u)\in D_2$; and if $u\in \zeta(\Upsilon_1)\cap \zeta(\Upsilon_2)$ then $f(u)\in M$.
        \item (fixed bundles are constant and distinct) for all $f,f'\in F$, $u\in \zeta(\Upsilon_1,\Upsilon_2)$ and $u',u''\in \pi_u\setminus \{u\}$, we have $f(u')=f'(u'')$ if and only if $u'=u''$; also, $f(u')\in P_{f(u)}$;
        \item (all pairs of free vertices have an embedding) for all $w_1\in X_1$ and $w_2\in X_2$, there is some $f\in F$ such that $f(\iota(\Upsilon_1))=w_1$ and $f(\iota(\Upsilon_2))=w_2$; and
        \item (free vertices map to free vertices) for all $f\in F$, we have $f(\iota(\Upsilon_1))\in X_1$ and $f(\iota(\Upsilon_2))\in X_2$.
    \end{enumerate}

     We call such an $F$ a \emph{realization} of $(\Upsilon_1, \Upsilon_2)$ with associated free vertices $\iota(F)=\iota(\Upsilon_1,\Upsilon_2)$.

     \begin{defn}
         We further extend the notion of a realization as follows. First, if $F=\{f\}$ where $f$ is a bijection from a subset $V_F\subset V(G)$ to some subset of $R_0$, then we also call $F$ a realization (but not of any particular blueprint pair) with $\iota(F)=\emptyset$. Next, given two realizations $F,F'$ with $V_F\cap V_{F'}=\emptyset$ we recursively define the realization $F\oplus F'$ with $\iota(F\oplus F')\coloneqq \iota(F)\cup \iota(F')$ as
    $$F \oplus F' \coloneqq  \{f\cup f' : f\in F, f'\in F'\}, $$
where if $f:V_F\to R$ and $f':V_{F'}\to R'$, then $f\cup f'$ is the function $\hat{f}:(V_F\cup V_{F'}) \to (R\cup R')$ defined by $$\hat{f}(x) = \begin{cases}
        f(x) & x \in V_F ,\\
        f'(x) & x \in V_{F'}.
    \end{cases}$$
    We also define $F\oplus \emptyset=F$ for all realizations $F$.
    
    Note that in our proof we will only be realizing a collection of pairwise non-overlapping blueprint pairs, and therefore $F\oplus F'$ will always be well-defined for realizations $F,F'$ of distinct pairs.
     \end{defn}

    Our algorithm works by realizing blueprint pairs and then chaining them together via the operation $\oplus$. Then, once we've generated some $F=F_1\oplus...\oplus F_k$ such that $c(F)$ is our desired subgroup (as we explain in \Cref{Sec.algorithm}, this will differ from $\Gamma_0$), we will extend $F$ by some bijection $f$ on the remaining vertices of $V(G)$ so that $F\oplus \{f\}:V(G)\to R_0$. Intuitively, we think of $\{f\}$ as a realization with multiplicity $\lambda=1$, so that free vertices have the same properties as fixed vertices.

\end{defn}

\begin{defn}\label{def.func_col}
    Given a group $\Gamma$, a vertex set $R$, an edge coloring $c : E(R) \to \Gamma$, a vertex coloring $\cC : R \to \Gamma$, a realization $F$, and some $h\in F$ (so that $h:V_F\to R$, with $V_F\subseteq V(G)$), define
    $$
   c(h)\coloneqq \sum_{u\in \iota(F)}d(u)\cdot \cC(h(u))+\sum_{\substack{x,y\in V_F \\ xy\in E(G)}}c(h(x)h(y))$$
   where the summation above is over unordered pairs of vertices $(x, y)$, and set
   \[c(F) \coloneqq \{c(f):f\in F\}.\]
Note that since fixed vertices (and therefore edges between them) are constant on $F$, for $F$ a realization of a blueprint pair $(\Upsilon,\Upsilon')$ at multiplicity $\lambda$, by \cref{colorful} we have
\[|c(F)|\geq \lambda.\]
\end{defn}

\begin{remark}\label{oplus}
     It is important to note that for two realizations $F,F'$ disjoint in their domains, our definition of $c(f\cup f')$ in general differs from $c(f)+c(f')$. However, notice that for any vertices $x\in V_F, y\in V_{F'}$ with $xy\in E(G)$, all of $F$ must be constant on $x$ and $F'$ must be constant on $y$ (it is impossible for a free vertex to have a neighbor that isn't in the same realization). Thus, there is some constant $C_{F,F'}$ such that $c(f)+c(f')+C_{F,F'} = c(f\cup f')$ for all $f\in F$ and $f\in F'$, and so $c(F\oplus F') = c(F) + c(F') + C_{F,F'}.$
\end{remark}

\section{The algorithm}\label{Sec.algorithm}
With the set of gadgets $K$ guaranteed by \Cref{K_partition}, in this section, we show how to find realizations of them one by one until we find a zero-sum copy of $G$ in $R_0$ under the coloring $c_0$. To do this, in \cref{sec:wllbehaved}, we provide the definition of $\kappa$-well-behaved colorings of subsets of $R_0$ (as we alluded to in \Cref{sec.example}) and equip these colorings with a suitable notion of size. In \Cref{phase0}, we prepare the parameters of the algorithm. We start with a $\kappa$-well-behaved tuple of minimal size, which allows us to define a new coloring $c$ on a subset of $R_0$. This will be the coloring that we use in the rest of the proof. 

Once we have finished these preparatory steps, in \Cref{sec:algo} we provide our algorithm for finding gadgets and updating the vertex set and coloring we are left with. In \Cref{embG}, we then verify that if the algorithm runs until the end, then it does indeed give a zero-sum copy of $G$. We will then verify the remaining assumptions about our algorithm in \Cref{gadget_counts,Sec.vertexcounts,sec.nofail}, which will complete the proof of our theorem.

\subsection{Well-behaved colorings}\label{sec:wllbehaved}

Now armed with all the necessary notions of blueprints, gadgets, and realizations, we are finally in a position to formally define what it means for a vertex subset $R\subset R_0$ to be $\kappa$--well-behaved, as we promised in \Cref{sec.example}.

\begin{defn}\label{well-behaved}
 For a vertex set $R$ and subgroup $\Gamma\leqslant\Gamma_0$, we say the pair $(R, \Gamma)$ is \textbf{$\kappa$--well-behaved} if there exist an $s\in \Gamma_0$, a subset $T\subset \Gamma_0$, and a vertex coloring $\mathcal{C}:R\to T$ satisfying:

    \begin{enumerate}[(i)]
        \item \label{C_correct}   $c_0(xy)+\Gamma= s + \mathcal{C}(x)+\mathcal{C}(y) + \Gamma$ for all $xy \in E(R_0)$;
        \item \label{T_order} $\kappa t\in \Gamma$ for all $t\in T$; and
        \item \label{R_size} $|R|\ge\begin{cases}
            |R_0|- 14\Delta^3(n-|\Gamma|) & \text{if } |\Gamma|\ge n/\alpha, \\ |R_0|/(\alpha\beta^{2\Delta})- 14\Delta^3(n-|\Gamma|) - 6\Delta^2(n/|\Gamma|)  & \text{if } |\Gamma|<n/\alpha.
        \end{cases}$

    \end{enumerate}

    In this case we say that $(R,\Gamma,T,\cC,s)$ forms a $\kappa$--well-behaved tuple of size $\sigma(R,\Gamma,T,\cC,s)\coloneqq |\Gamma|$ (as shown in \Cref{fig:k-well-behaved}). 
\end{defn}

\begin{figure}[htbp]
\captionsetup{justification=centering}
  \begin{center}
  \begin{tikzpicture}[
  v2/.style={fill=black,minimum size=4pt,ellipse,inner sep=1pt},
  node distance=1.5cm,scale=0.48]
  \node[v2] (v1) at (4.2,  -2) {};  
    \node[v2] (v2) at (5.8, -2) {};  
    \node[v2] (v3) at (5, -6) {}; 
  \draw (0, 0)--(10, 0);
  \draw (0, 0)--(0, -8);
  \draw (10, -8)--(10, 0);
  \draw (0, -8)--(10, -8);
   \draw (5,-4) circle [radius=3.7cm];
   \draw (5,-2) circle [radius=1.6cm];
   \draw (5,-6) circle [radius=1.1cm];
  \draw (v1)--(v2);  
\draw (v1)--(v3);   
\draw (v2)--(v3);
\node[font=\scriptsize,left] at (v1) {$t_1$};
\node[font=\scriptsize,xshift=6pt,yshift=-5pt] at (v1) {$u$};
\node[font=\scriptsize,right] at (v2) {$t_1$};  
\node[font=\scriptsize,xshift=-6pt,yshift=-5pt] at (v2) {$v$};
\node[font=\scriptsize,below] at (v3) {$t_2$}; 
\node[font=\scriptsize,right] at (v3) {$w$};
\node[font=\scriptsize,above] at (5,-2) {$s+2t_1$}; 
\node[font=\scriptsize,left] at (4.5,-4) {$s+t_1+t_2$}; 
\node[font=\scriptsize,right] at (5.5,-4) {$s+t_1+t_2$}; 
\node[font=\scriptsize,left] at (10,-7.5) {$R_0$}; 
\node[font=\scriptsize,left] at (8,-6) {$R$}; 
\end{tikzpicture}
\end{center}
\caption{An example of a $\kappa$--well-behaved vertex set, where the labels of vertices represent their color under $\cC$ and labels of edges represent  $s+\cC(x)+\cC(y)$ for $x,y\in \{u,v,w\}$.}\label{fig:k-well-behaved}
\end{figure}
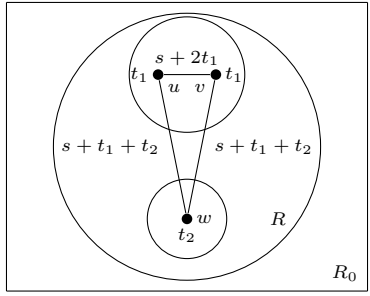

\begin{remark}\label{T_size}
    Observe that in the above definition we can always assume that $|T|\cdot |\Gamma|\leq n$. Indeed, for all $x,y$ such that $\cC(x)+\Gamma=\cC(y)+\Gamma$ we can assume that $\cC(x)=\cC(y)$ since condition \ref{C_correct} will remain satisfied.
\end{remark}

Let us briefly elucidate on \cref{R_size}. Intuitively, we want to ensure that our remaining pool of vertices is large enough, but we also need to take into account that the pool of vertices shrinks as we go along the process. Our algorithm will work by systematically shrinking the group we work with until we reach the trivial group $\{0\}$.
Hence, a tuple with a smaller $|\Gamma|$ allows us to jump further ahead in the process so we want to reward this by allowing it to lose more vertices. In line with this, the terms in \cref{R_size} are representative of potential sources of vertex loss in the process: the jump from $|\Gamma|\ge n/\alpha$ to $|\Gamma|<n/\alpha$ may divide the vertex pool by $\alpha\beta^{2\Delta}$, for each gadget we extract we lose an order of $\Delta^3$ vertices, and for each $t\in T$ we have to regularly throw away any class $\cC_t\coloneqq \{x\in R:\cC(x)=t\}$ if it becomes too small to work with; our threshold will be $2\Delta^2$.

With this definition, we can give a general overview of our algorithm. It is divided into three phases. First, we begin by pointing out that we can assume there is no monochromatic copy of $G$ in $E(R_0)$. This fact will become useful in Phase 1. 

In Phase 0, we complete the setup. We look for a $\kappa$--well-behaved vertex set of minimum size $\sigma$. Then, depending on $\sigma(R,\Gamma,T,\cC,s)$, we will either go to Phase 1 or straight to Phase 2. In both cases, the nature of $\cC$ will inform how we look for gadgets.

In Phase 1, we begin by taking the largest class of vertices $\cC_i\coloneqq \{x\in R:\cC(x)=i\}$ and setting that to be our new set of vertices (updating the coloring appropriately), thereby returning us to a state the looks a lot like the trivial well-behaved case (but losing up to a constant factor on the size of $R$). Then, we will realize blueprint pairs from $K$ at multiplicity $\beta$ one by one until we have generated a coset $a+H$ of $\Gamma$. We will then preform an operation that will reduce our coloring to one on $\Gamma/H$, and proceed to realize blueprint pairs with this coloring. This will repeat until we reach a group $\Gamma/H'$ with size smaller than $|\Gamma_0|/\alpha$, at which point we move on to Phase 2.

In Phase 2, much like in Phase 1, we will realize blueprint pairs from $K$, but now it will be at multiplicity $2$. The rest is exactly the same as Phase 1: we find a coset, quotient our operating group by it, and proceed, until we have reached the trivial group.

Then, to finish, we will take our set of functions $F$ and extend them uniformly so that each embeds a copy of $G$ into $R_0$. We will aim to conclude that for each $g\in \Gamma$ there will be some function $f_g\in F$ such that $\sum_{xy\in E(G)} c_0(f_g(x)f_g(y))=g$, and thus we will have an $f_0$ that witnesses a zero-sum copy of $G$ in $R_0$.

\subsection{Phase 0}\label{phase0}

As described in the previous section, Phase 0 is quite simple. Let $(R,\Gamma,T,\cC,s)$ be a $\kappa$--well-behaved tuple of minimum size $\sigma$.  Note that $(R_0,\Gamma_0,\{0_{\Gamma_0}\},0_{id},0_{\Gamma_0})$ is a $\kappa$--well-behaved tuple of size $n$, and so the minimum is well-defined. Let $\cC_i\coloneqq \{x\in R:\cC(x)=i\}$ for each $i\in T$. Observe that for any $r\in T$, $(R,\Gamma,T-r,\cC-r,s+2r)$ is also $\kappa$--well-behaved. Hence, we may assume that the largest such $\cC_i$ is $\cC_0=\cC_{0_{\Gamma_0}}$.

If $|\Gamma|\geq n/\alpha$, then we begin by defining $R_1\coloneqq \cC_0$. Then on $R_1$, we have $c_0(xy)+\Gamma= s'+\Gamma$ for all $x, y$, and so we define a new coloring \beq{newcol}c(xy)\coloneqq c_0(xy)-s = c_0(xy)-s-\cC(x)-\cC(y) \in \Gamma\enq 
for all $x, y\in R$.
Then, run \nameref{sub1} (described later) with input $(R_1,\Gamma,c,\cC)$. Observe that $\cC$ is by definition identically $0_{\Gamma_0}$ on $R_1$.

On the other hand, if $|\Gamma|<n/\alpha$, then we again define $c(xy)\coloneqq  c_0(xy)-s-\cC(x)-\cC(y)\in\Gamma$ like before, but now set $$R_1=\bigcup_{i\in T: |\cC_i|\geq 2\Delta^2}\cC_i$$ and run \nameref{sub1} with input $(R_1,\Gamma,c,\cC)$.

We denote the two cases when $|\Gamma|\geq n/\alpha$ and $|\Gamma|< n/\alpha$ by Phase 1 and Phase 2 respectively. In what follows, the only difference is that in Phase 1 we look for realizations at multiplicity $\beta$ and in Phase 2 we look for them at multiplicity 2 instead.

\begin{remark}\label{R'R_size}
    We note here the difference in size between $R$ and $R_1$ for the two cases:
    \begin{enumerate}
        \item For Phase 1: $|R_1|>|R|/|T|\geq |R|/\alpha$.
        \item For Phase 2: $|R_1|>|R|-2\Delta^2|T| \geq |R| - 2\Delta^2 n.$
    \end{enumerate}
\end{remark}
\subsection{The main algorithm}\label{sec:algo}

Given an abelian group $\Gamma$, a subgroup $H'$ of $\Gamma$ and a subgroup $H\subset \Gamma/H'$ such that $H=\{a_1+H',a_2+H',\ldots,a_k+H'\}$ for some $a_1,\ldots,a_k\in \Gamma$, define 
\[\psi(H,H')\coloneqq \{a_i+h:i\leq k\text{ and } h\in H'\},\]
and note that this is independent of the choice of coset representatives $a_i$ and that $\psi(H, H')\leqslant\Gamma$.

Recalling from \Cref{sec.blueprints} the definitions of $\alpha$, $\beta$ (\cref{equ:alphabeta}), $G'$ (\cref{G'def}), $P$ (\cref{equ:P}),  and $K$ (\Cref{K_partition}), we now introduce our algorithms. \nameref{sub1} takes as input a tuple $(R, \Gamma, c, \cC)$ consisting of a set of vertices, a subgroup of $\Gamma_0$, an edge-coloring $E(R)\to \Gamma$ and a vertex coloring $R\to \Gamma_0$, respectively, and returns a realization $F$ with $V_F=V(G)$. \nameref{Amain} takes as input a tuple $(R, \Gamma, c, \cC, \lambda, K)$ with $R, \Gamma, c, \cC$ as before, $\lambda\geq 2$ an integer and $K$ a set of (pairwise non-overlapping) blueprint pairs, and returns a tuple consisting of a coset of some subgroup $H\leqslant \Gamma$, the subgroup itself, a subset $R'\subset R$ of unused vertices, a partial map from a family of blueprint pairs in $K$ to the pool of vertices, and a set $K'\subset K$ of unused blueprint pairs.

\begin{algorithm}[H]
\caption{Embedding-Algorithm
}\label{sub1}
\begin{algorithmic}[H]
\State \textbf{Input:} $[R,\Gamma,c,\cC]$ \Comment{The pool of vertices, group, edge coloring, and vertex coloring.}  
\State \textbf{Initialize } $i = 1$, $R_1 = R$, $\Gamma_1 = \Gamma$, $m_1=|\Gamma|$, $c_1=c$, $H'_1 = \{0_\Gamma\}$, $K_1=K$.

\While{$m_i>1$}
\If{$m_i\geq \frac{n}{\alpha}$} 
    \State $\lambda=\beta$ \Comment Phase 1
\Else
    \State $\lambda=2$ \Comment Phase 2
\EndIf    
\State $(A_i,H_i,R_{i+1},F_i,K_{i+1})\gets$ \nameref{Amain} ($R_i,\Gamma_i,c_i,\cC,\lambda,K_i$)
\Comment{A tuple: the subgroup coset we found, the subgroup, the remaining vertices, a partial map from $G'$ to $R_i$, and the remainder of $K$}
    \State $H_{i+1}' \gets \psi(H_i,H_i')$
    \State $\Gamma_{i+1} \gets \Gamma/H_{i+1}',$ \quad $m_{i+1} \gets |\Gamma_{i+1}|,$ \quad $c_{i+1}\gets c_1+H_{i+1}'$ 
    \State $i\gets i+1$
\EndWhile
\State $\overline R \gets R_i$
\State Let $\overline{\Pi} = P \cup \left(\bigcup_{(\Upsilon,\Upsilon')\in K_{i}}(\{\iota(\Upsilon,\Upsilon')\}\cup \{\pi_v:v\in \zeta(\Upsilon,\Upsilon')\})\right)$\Comment{We want to now embed the rest of the vertices: those not assigned to blueprints, and those in blueprints that we didn't end up using}
\For{each set $\pi\in \overline\Pi$}
\State For each $u\in \overline R$ in turn, define $\cC_u\coloneqq \{y\in \overline R:\cC(y)=\cC(u)\}$
\State Take any $w \in \overline R$ such that $|\cC_w|\geq |\pi|$
\State Take any $\hat{\pi}\subset \cC_w$ of size $|\pi|$ and any bijection $f'_\pi: \pi\to \hat{\pi}$, and set $F'_\pi = \{f'_\pi\}$
\State $\overline R \gets \overline R\setminus \hat{\pi}$
\EndFor 
\State \textbf{Define} $F_{i+1}=\bigoplus_{\pi\in \overline\Pi}F'_\pi$ 

\Comment{Now every vertex of $G$ that doesn't appear in a used blueprint pair belongs to the domain of $F_{i+1}$}

\State Let $F=\bigoplus_{j\leq i+1} F_j$
\Comment{$F$ consists of injective mappings from $V(G)$ to $R$}
\State \Return $F$

\end{algorithmic}
\end{algorithm}
\begin{algorithm}[H]\caption{Realization-Finder}\label{Amain}
\begin{algorithmic} 
\State \textbf{Input:} [$R,\Gamma,c,\cC,\lambda, K$]
\Comment{The vertex set, group, edge coloring, vertex coloring, prescribed multiplicity for the gadgets, and current pool of blueprint pairs.}
\State \textbf{Define}: \quad $F\gets \emptyset,$ \quad $i=1$  \Comment{$F$ is the realization we've found}
\While{$c(F)$ does not contain a coset $A$ of some nontrivial subgroup $H\leqslant\Gamma$}
\State Take any $(\Upsilon_i,\Upsilon_i')\in K$ and set $K\gets K\setminus\{(\Upsilon_i,\Upsilon_i')\}$
\State\textbf{Find} a realization of
$(\Upsilon_i,\Upsilon_i')$ at multiplicity $\lambda$ in $R$ with edge-coloring $c$ and vertex-coloring $\cC$.
\State Let the realization be $F'_i$ with gadget $\gamma_i$
\State $F \gets F\oplus F'_i,$ \quad $R \gets R \setminus V(\gamma_i)$
\State Let $L=\{v\in R: |\{w\in R: \cC(w)=\cC(v)\}|<2\Delta^2\}$
\State $R \gets R\setminus L$ \Comment We remove all the $\cC_i$ that have gotten too small.
\State $i\gets i+1$
\EndWhile
\State Let $A$ be an arbitrary coset of some nontrivial subgroup of largest size $H\leqslant \Gamma$ contained in $c(F)$
\State \Return ($A,H,R,F,K$)

\end{algorithmic}
\end{algorithm}

\subsection{Finding a zero-sum embedding of $G$}\label{embG}

Now we  show that if \nameref{sub1} terminates successfully, then there is  a zero-sum copy of $G$ in $E(R_0)$.

\begin{theorem}
    If \nameref{sub1} running with the input described in \Cref{phase0} successfully returns a family of functions $F$, then there is some $f_0\in F$ that witnesses a zero-sum embedding of $G$ into $E(R_0)$ under coloring $c_0$. 
\end{theorem}

\begin{proof}

If \nameref{sub1} returns a set of functions $F$, then the algorithm successfully terminated and so there is a sequence of realizations $F_1,\ldots,F_k$ such that for each $1\leq i\leq k$, $c_i(F_i)$ contains a coset $a_i+H_i \subset \Gamma/H_{i}'$, as well as the realization $F_{k+1}=\bigoplus_{\pi\in \overline\Pi}F'_\pi$ obtained at the end of \nameref{sub1}. Let  $r_i\in \Gamma$ be so that $a_i=r_i+H'_i$.

We first prove that 
\begin{equation}\label{Fgenall}
    c_1(F_1 \oplus ... \oplus F_k) = \Gamma.
\end{equation}
For this, we use the following claim:

\begin{claim*}
For all $\ell \in [k]$, there exists some $b_\ell\in\Gamma$ such that
$$c_1(F_1\oplus\cdots \oplus F_{\ell})\supset b_\ell+H'_{\ell+1}.$$
\end{claim*}
\begin{subproof}[Proof of claim]
We proceed by induction. We have $c_1(F_1)\supset r_1+H_1=r_1+H'_2$.

Now suppose the claim holds for some $\ell \in [k-1]$. Recalling \Cref{oplus}, we define a constant $r_{\ell,\ell+1}\in \Gamma$ such that 
$$c_1(F_1\oplus \cdots \oplus F_{\ell+1}) = c_1(F_1\oplus \cdots \oplus F_\ell) + c_1(F_{\ell+1}) + r_{\ell,\ell+1}.$$

We have $c_{\ell+1}(F_{\ell+1})\supset A_{\ell+1}=a_{\ell+1} + H_{\ell+1}$, $c_{\ell+1}=c_1/H'_{\ell+1}$ and $H_{\ell+1}\leqslant \Gamma/H'_{\ell+1}$ by construction. Thus, there is some $I_{\ell+1}\subset \Gamma$ such that $H_{\ell+1} = \{a+H_{\ell+1}':a\in I_{\ell+1}\}$, and so $c_{1}(F_{\ell+1})\supset \{r_{\ell+1}+a+h_a:a\in I_{\ell+1}\}$ for some collection $\{h_a\}_{a \in I_{\ell+1}}\se H_{\ell+1}'$. Using our induction hypothesis, we then have
$$\begin{aligned}
    c_1(F_1\oplus\cdots \oplus F_\ell\oplus F_{\ell+1})&=c_1(F_1\oplus \cdots \oplus F_\ell) + c_1(F_{\ell+1}) + r_{\ell,\ell+1}\\
    &\supset r_{\ell,\ell+1} +b_\ell  + r_{\ell+1} + H_{\ell+1}'+ \{a+h_a:a\in I_{\ell+1}\}\\
    & = r_{\ell,\ell+1} +b_\ell + r_{\ell+1} + \psi(H_{\ell+1},H_{\ell+1}')\\
    &=r_{\ell, \ell+1} +b_\ell + r_{\ell+1} + H_{\ell+2}'.
\end{aligned}$$ 
This proves the claim.
\end{subproof}

Now, as the algorithm terminated with $m_{k+1}=1 = |\Gamma/H'_{k+1}|$ it must be the case that $H'_{k+1} = \Gamma$. Hence, by the claim $c_1(F_1\oplus\cdots \oplus F_{k}) = \Gamma$, as desired.

Next, recalling that $c_0(xy)=c_1(xy)+s+\cC(x)+\cC(y)$ for all $x, y$ by \cref{newcol}, the definition of $G'$ in \cref{G'def} and its partition given by \cref{eq.G'partition}, and by \Cref{def.func_col}, we have for all $f\in F$:
$$
\begin{aligned}
\sum_{xy\in E(G)}c_0(f(x)f(y))&=c_0(f) - \sum_{v\in \iota(F)} \cC(f(v)) \cdot d(v)\\
&= c_1(f) + s\cdot n + \sum_{v\in V(G)\setminus \iota(F)} \cC(f(v)) \cdot d(v) + \sum_{v\in \iota(F)} \cC(f(v)) \cdot d(v) - \sum_{v\in \iota(F)} \cC(f(v)) \cdot d(v) \\
&= c_1(f) + 0_{\Gamma_0} + \sum_{\pi\in \Pi}\sum_{v\in \pi} \cC(f(v)) \cdot d(v).
\end{aligned}$$
  
For each $\pi\in \Pi$ we fix some arbitrary $v_\pi \in \pi$, and note that for all $v \in \pi$ we have $\cC(f(v)) = \cC(f(v_\pi))$. Recalling \Cref{well-behaved} we also know $\kappa \cdot \cC(f(v_\pi)) \in \Gamma$ for all $\pi \in \Pi$. Finally, observe that by construction $\sum_{v\in \pi}  d(v)$ is divisible by $\kappa$. Hence
$$\sum_{xy\in E(G)}c_0(f(x)f(y))=c_1(f) + \sum_{\pi\in \Pi}\cC(f(v_\pi)) \cdot \sum_{v\in \pi}  d(v) = c_1(f) + \sum_{\pi\in \Pi} g_\pi,$$
for some $\{g_\pi: \pi\in \Pi\}\subset \Gamma$ independent of $f$ because, by definition of a realization, $F$ is constant on each $v\in V(G)\setminus \iota(F)$. Thus, we have some constant $g'\in \Gamma$ such that for all $f\in F$,
$$\sum_{xy\in E(G)}c_0(f(x)f(y))=c_1(f)+g'.$$

By \cref{Fgenall}, there exists some $f\in F$ with $c_1(f)=-g'$, and so $f$ provides a zero-sum embedding of $G$ in $R_0$, as desired.
\end{proof}

Then, to prove \Cref{main}, we only need to show that \nameref{sub1} always successfully terminates. For this we need to show that for some $k$ with $m_{k+1}=1$ the following holds.
\begin{enumerate}[(I)]
    \item $K_i\neq \emptyset$ for $1\leq i\leq k$.\label{cond:K}
    \item $R_i\neq\emptyset$ for $1\leq i\leq k+1$.\label{cond:R}
    \item In \nameref{sub1}, for each $\pi\in \overline \Pi$ there exists some $\cC_w$ satisfying $|\cC_w|\geq |\pi|.$\label{cond:embed}
    \item A realization of $(\Upsilon_i,\Upsilon_i')$ at multiplicity $\lambda$ is always successfully found when running \nameref{Amain}.\label{cond:realize}
\end{enumerate}

Part \ref{cond:K} will be proven in \Cref{gadget_counts} by \Cref{K_size}. For part \ref{cond:embed}, note that \nameref{Amain} guarantees that each non-empty $\cC_w$ satisfies $|\cC_w|\ge 2\Delta^2\ge|\pi|$ for all $\pi\in \overline \Pi$. By pigeonhole, it then suffices to show that $|\overline{R}|\geq 4\Delta^2 \cdot|\overline \Pi|$. Since $|\overline\Pi|\leq 2n$, we will prove that $|\overline{R}|\geq 8\Delta^2n$; in fact in \Cref{vertex_count_thm} we prove this lower bound holds for every $R_i$ (and so in particular for $\overline{R}=R_{k+1}$), and hence handle both parts \ref{cond:R} and \ref{cond:embed} above. Finally, part \ref{cond:realize} will be shown in \Cref{phase_1,phase_2}, where we distinguish two cases depending on the multiplicity $\lambda$ at which we run \nameref{Amain}.

\section{Gadget counts}\label{gadget_counts}
\begin{theorem}\label{K_size}

No set $K_i$ is ever empty during the execution of \nameref{sub1}.

\end{theorem}

\begin{proof}
Suppose we have successfully gone through $\ell$ runs of the algorithm. We show that $K_{\ell+1}\neq\emptyset$. Say that the first $\ell_0$ calls of \nameref{Amain} in the running of \nameref{sub1} are in Phase 1, and the calls after $\ell_0$ are in Phase 2.

Recall that we start by looking for a $\kappa$--well-behaved tuple $(R,\Gamma,T,\cC,s)$ of minimal size $\sigma=|\Gamma|$, and call \nameref{sub1} with the group $\Gamma$. 

Write $S_{i}$ to denote the set $c(F)$ used in the $i^\text{th}$ call of \nameref{Amain}, so that each $S_i$ contains a coset $A_i=a_i+H_i$. On the $(\ell+1)^\text{th}$ call, we start with $S_{\ell+1}=\emptyset$ and attempt to incrementally add elements of $\Gamma/H'_{\ell}$ to it until it contains a coset $A_{\ell+1}=a_{\ell+1}+H_{\ell+1}$.

We will give an upper bound on the number of blueprint pairs used in call $i$ (i.e. the quantity $|K_i|-|K_{i+1}|$) for each $1\leq i \leq \ell$. Combined with the lower bound on $K_1=K$ given by \Cref{K_partition}, this will yield the theorem. 

Note that the maximum size a subset of $\Gamma/H'_i$ can have without containing a coset of $H_i\leqslant \Gamma/H'_i$ is $\frac{|H_i|-1}{|H_i|} \cdot |\Gamma/H'_i|=|\Gamma/H'_{i}|-|\Gamma/H'_{i+1}|$. 

First, for $1\leq i\leq \ell_0$, note that by \Cref{Kneser} every blueprint pair from $K_i$ that is selected in \nameref{Amain} increases the size of $c(S_i)$ by at least $\beta-1$, and so we use at most
\[|K_i|-|K_{i+1}|\leq \frac{|\Gamma/H'_{i}|-|\Gamma/H'_{i+1}|}{\beta-1}+1\] 
blueprint pairs to generate $S_i$.

For $i>\ell_0$, we instead know by \Cref{Kneser} that every blueprint pair from $K_i$ that is selected increases the size of $c(S_i)$ by at least one, and so
\[|K_i|-|K_{i+1}|\leq |\Gamma/H'_{i}|-|\Gamma/H'_{i+1}|+1.\]
Recall that $H'_1=\{0_\Gamma\}$. Putting together the above two bounds (noting the second is vacuous if $\ell\le\ell_0$), we see that

\begin{align*}
\label{equ:gadgetcounts}|K_1|-|K_{\ell+1}|&\leq \sum_{i=1}^{\ell_0}\left(\frac{|\Gamma/H'_{i}|-|\Gamma/H'_{i+1}|}{\beta-1}\right)+\ell_0+\sum_{i=\ell_0+1}^{\ell}\left(|\Gamma/H'_{i}|-|\Gamma/H'_{i+1}|\right)+(\ell-\ell_0) \\ &\leq \frac{|\Gamma|}{\beta-1} + \frac{\beta-2}{\beta-1}|\Gamma/H'_{\ell_0+1}|+\ell \\
&\leq \frac{n}{\beta-1}+|\Gamma/H'_{\ell_0+1}| +\ell.
\tag{$**$}
\end{align*}

Recall now that we only stop running Phase 1 when the size of the subgroup we generate is less than $ n/\alpha$, i.e. we have $|\Gamma/H'_{\ell_0+1}|< n/\alpha$. Moreover, we clearly have $\ell\leq \log_2 |\Gamma|$ since at every iteration we quotient by a nontrivial subgroup of $\Gamma$. Finally, recalling that $\beta=2\alpha$ we see that
$$|K_1|-|K_{\ell+1}|\leq \frac{n}{\beta-1}+\frac{n}{\alpha} +\log_2(n)< \frac{2n}{\alpha}=\frac{n}{5\Delta^6}=|K_1|$$ and therefore $K_{\ell+1}\neq\emptyset$ as desired. Moreover, the same analysis shows that in the process of running the $(\ell+1)^\text{th}$ call of the algorithm, $K_{\ell+1}$ stays non-empty.
\end{proof}

\section{Vertex counts}\label{Sec.vertexcounts}

Recall the discussion at the end of \Cref{Sec.algorithm}, which shows that in order to prove points \Cref{cond:R,cond:embed} it suffices to prove that $|R_i|\geq 8\Delta^2n$ for each $i$ in the running of our algorithm.

\begin{lemma} \label{vertex_step}
Suppose we call \nameref{Amain}  $(R_i,\Gamma_i,c_i,\cC,\lambda, K_i)$ and it returns $(A_i,H_i,R_{i+1},F_i,K_{i+1})$. Then $|R_{i+1}|\geq |R_i| - 7\Delta^3 (|\Gamma_i|-|\Gamma_i|/|H_i|).$
\end{lemma}

\begin{proof}
    For each blueprint pair $(\Upsilon,\Upsilon')$ we realize, we remove $|V(\gamma_j)| + |L|$ vertices from $R_i$. Now,
    \[|V(\gamma_j)|\le 2\lambda -2 + (d_1+d_2)\cdot \kappa\leq 2(\lambda-1+\Delta^2)\]
    and
    \[|L|=\big|\{v\in R_i\setminus V(\gamma_j):|\{w\in R_i\setminus V(\gamma_j):\cC(w)=\cC(v)\} | <2\Delta^2\}\big|\leq (2\Delta^2-1) \cdot (2\lambda-2+d_1+d_2),\] 
    where the latter bound follows since we know that all the vertex classes under $\cC$ in $R_i$ have size at least $2\Delta^2$, and each of the vertices in $P_v$ by definition are in the same class as $v$ so $V(\gamma_j)$ only contains vertices from at most $2\lambda-2+d_1+d_2$ different classes. Thus, we use at most $(4\Delta^2-2)(\lambda-1 +\Delta) + 2(\lambda - 1 +\Delta^2) < 4\Delta^2(\lambda-1+\Delta)+2\Delta^2$ vertices.

    Now recall from \Cref{gadget_counts} that we only use at most $|\Gamma_i|\cdot(1-1/|H_i|)/(\lambda -1)$ blueprint pairs when generating $H_i$. Then in total, we use less than $7\Delta^3(|\Gamma_i|-|\Gamma_i|/|H_i|)$ vertices, and so we indeed have that $|R_{i+1}|>|R_i|-7\Delta^3(|\Gamma_i|-|\Gamma_i|/|H_i|).$
\end{proof}

\begin{corollary}\label{vertex_jump}
    When running \nameref{sub1}, we have $|R_{i_2}|> |R_{i_1}| -7\Delta^3 (|\Gamma_{i_1}| - |\Gamma_{i_2}|)$ for every $i_1\le i_2$.
\end{corollary}

\begin{proof}
    By \Cref{vertex_step}, we have that for each $i_1\leq i< i_2$, 
    \[|R_{i+1}|> |R_i| - 7\Delta^3 (|\Gamma_i|-|\Gamma_i|/|H_i|) = |R_i| - 7\Delta^3 (|\Gamma_i|-|\Gamma_{i+1}|),\]
    and the result follows.
\end{proof}

\begin{theorem}\label{vertex_count_thm}
    Every set $R_i$ obtained during the execution of \nameref{sub1} satisfies $|R_i|\geq 8\Delta^2 n$.
\end{theorem}

\begin{proof} Suppose we have successfully gone through $\ell$ runs of the algorithm. We show that $|R_{\ell+1}|\geq 8\Delta^2 n$.

By \Cref{vertex_jump}, we have that $$|R_{\ell+1}|>|R_1|-7\Delta^3(|\Gamma_1|-|\Gamma_{\ell+1}|) > |R_1|-7\Delta^3 n.$$

Recall also that $|R_0| =  \Delta^{42\Delta^6}\cdot C'(\alpha\beta,\Delta) \cdot n > 100\alpha\beta^{2\Delta}\Delta^{2\beta} \cdot n$. It remains to lower bound $|R_1|$ relative to $|R_0|$. Say the $\kappa$--well-behaved tuple we found in Phase 0 (\Cref{phase0}) was $(R,\Gamma,T,\cC,s)$.

If $|\Gamma|\ge n/\alpha$, then by \Cref{well-behaved} we have $$|R|\geq |R_0|-14\Delta^3(n-n/\alpha)$$ and by \Cref{R'R_size} we have 
\[|R_1|\geq |R|/\alpha.\]

Putting everything together, this yields
$$|R_{\ell+1}|>|R_1|-7\Delta^3 n>|R|/\alpha - 7\Delta^3 n \ge (|R_0|-14\Delta^3(n-n/\alpha))/\alpha - 7\Delta^3 n > 8\Delta^2 n.$$

On the other hand, if instead we have $|\Gamma|<n/\alpha$, then by \Cref{well-behaved} and \Cref{R'R_size} we have
$$|R_{\ell+1}|>|R_1|-7\Delta^3n > |R|-(7\Delta^3+2\Delta^2) n \ge |R_0|/(\alpha\beta^{2\Delta}) - (21\Delta^3 + 8\Delta^2) n > 8\Delta^2 n ,$$
as desired. This completes the proof, and hence shown parts \ref{cond:R} and \ref{cond:embed}.
\end{proof}

Let us note that this is not the part of the proof of \Cref{main} that drives up the value of $C(\Delta)$, hence why the bounds in \Cref{vertex_count_thm} are so loose. Instead, the value for $C(\Delta)$, and most importantly its dependence on $C'(\alpha\beta,\Delta)$, comes from the argument at the end of \Cref{phase_1}.

\section{Realizing blueprints does not fail}\label{sec.nofail}

We will now show that every instance of \nameref{Amain}, called as a subroutine in Phase 1 and Phase 2 of \nameref{sub1}, always successfully realizes the desired blueprint pairs. For Phase 1 we do this by arguing that either we have a monochromatic copy of $G$ (which gives us a contradiction since we assumed no such copy exists at the start), or we can draw a contradiction to the minimality of $\sigma$ (chosen in \Cref{phase0}). Phase 2 is  similarly  shown to always succeed by drawing a contradiction to the minimality of $\sigma$.

We begin with two lemmas that will be of use in both cases.

\begin{lemma}\label{prelim_phase_arg}
    Let $X\subset R_0$ be a set of vertices along with an edge coloring $c=c_1/H':E(X)\to \Gamma/H'$ for some $H'\leqslant \Gamma\leqslant\Gamma_0$ and a vertex coloring $\cC:X\to \Gamma_0$ such that $\cC_t\coloneqq \{v\in X:\cC(v)=t\}$ has $|\cC_t|\geq 2\Delta(\lambda-2+\Delta)$ for all $t$ with $\cC_t\neq \emptyset$. 
    
    Suppose we fail to find a realization for some blueprint pair $(\Upsilon,\Upsilon')$ of type $(d, d', m)$ at multiplicity $\lambda$. Then there is no $(d,\lambda)$--,  $(d',\lambda)$-- or $(d,d',m,\lambda)$--gadget in $X$.
\end{lemma}

The above lemma should be thought of as saying that if we fail to realize a blueprint, then it is because we failed to find a gadget satisfying \cref{colorful} in \Cref{def:gadgets}, and not because we could not satisfy \cref{Bundles_correct} in \Cref{def:realization} (recall that a gadget is allowed to have $P_v=\emptyset$ for all $v\in D_1\cup D_2\cup M$).

\begin{proof}
    Let us argue by contradiction by supposing that $\gamma$ is a $(d,\lambda)$--,  $(d',\lambda)$-- or $(d,d',m,\lambda)$--gadget in $X$ with vertex set $V(\gamma)=D_1\cup X_1\cup X_2\cup D_2 \cup M\cup (\bigcup_{v\in D_1\cup D_2\cup M}P_v)$, but nevertheless, there is no realization $F$ of $(\Upsilon,\Upsilon')$.

    First, define $V\coloneqq D_1\cup D_2\cup M$ and observe that
    \[|\zeta(\Upsilon, \Upsilon')\setminus \zeta(\Upsilon')|=|D_1|, \quad |\zeta(\Upsilon, \Upsilon')\setminus \zeta(\Upsilon)|=|D_2|, \quad \text{ and } \quad |\zeta(\Upsilon)\cap \zeta(\Upsilon')|=|M|,\]
    so there exist bijections between each of these three pairs of sets. Since $D_1, D_2$ and $M$ are pairwise disjoint, let the union of the above three bijections be the bijection $f:\zeta(\Upsilon, \Upsilon') \to V$.
    
    Then, for each $v\in V$, observe that $\cC_{\cC(v)}$ is non-empty, and therefore has size at least $2\Delta(\lambda-2+\Delta)\ge (2\lambda-4 +d+d')\cdot \kappa$. Then, it follows that
    \begin{align*}
        |\cC_{\cC(v)}\setminus (V\cup X_1\cup X_2)|&\geq (2\lambda-4 +d+d')\cdot \kappa - (2\lambda-2+d+d'-m)\geq (d+d'-m-2)\cdot (\kappa-1)+(2\lambda-2)\kappa\\&\geq (\kappa-1)|V|.
    \end{align*}
    
    Now, for each $v\in V$ 
    fix some $P'_v\subset \cC_{\cC(v)}\setminus (V\cup X_1\cup X_2)$ of size $|\pi_{f^{-1}(v)}|-1$ such that the sets $\{P'_v: v\in V\}$ are pairwise disjoint; this is possible since $|\pi_{f^{-1}(v)}|-1\leq \kappa - 1$ for each $v\in V$. Then define $\gamma'$ with vertex set $V\cup X_1\cup X_2\cup(\bigcup_{v\in V}P'_v)$ and $E(\gamma')=E(\gamma[V\cup X_1\cup X_2])$. Then $\gamma'$ is a gadget of the same type as $\gamma$.
    
    We now find a realization $F$ of $(\Upsilon,\Upsilon')$ using $\gamma'$. For each $u\in \zeta(\Upsilon,\Upsilon')$ we know that $|P'_{f(u)}|=|\pi_{u}|-1$, so there exists a bijection $f_u:\pi_{u}\setminus\{u\} \to P'_{f(u)}$. Then, for every pair $(x_1,x_2)\in X_1\times X_2$ define the map $f_{(x_1,x_2)}: V(\Upsilon,\Upsilon')\cup \left(\bigcup_{v\in \zeta(\Upsilon,\Upsilon')}\pi_v\right)\to V(\gamma')$ as
    $$f_{(x_1,x_2)}(u)=\begin{cases}
        f(u) & \text{if } u\in \zeta(\Upsilon,\Upsilon')
        \\ f_{u'}(u) & \text{if } u\in \pi_{u'} \text{ for some } u'\in\zeta(\Upsilon, \Upsilon')
        \\ x_1 & \text{if } u=\iota(\Upsilon)
        \\ x_2 & \text{if } u=\iota(\Upsilon')
    \end{cases}$$
and let $F=\{f_{(x_1,x_2)}:(x_1,x_2)\in X_1\times X_2\}$. It is now easy to see that $F$ satisfies all the conditions of \Cref{def:realization}, and so $F$ is a realization of $(\Upsilon,\Upsilon')$, thereby yielding a contradiction.
\end{proof}

\begin{lemma}\label{core_phase_arg}
    Suppose that $d, d'\leq\Delta$ and $m\leq \min(d, d')-1$ are positive integers such that $\gcd(d, d', n)=\kappa$ and that   
    $X\subset R_0$ is some pool of vertices. Let $c=c_1/H':E(R_0)\to \Gamma/H'$ (for some $H'\leqslant \Gamma$) be an edge coloring and $\cC:R_0\to \Gamma_0$ a vertex coloring that is constant on $X$. 
    
    If there are no $(d,\lambda)$--, $(d',\lambda)$-- or $(d,d',m,\lambda)$--gadgets under the colorings $c$ and $\cC$, then there exist sets $I\subset X$, $S=\{s_1,\ldots,s_{\lambda-1}\}\subset \Gamma$, $T\subset \Gamma$, and a vertex coloring $\cC':I\to T$ such that: 
    \begin{enumerate}
        \item $|I|\geq |X|/(\lambda-1)^{d+d'} - d-d'$;
        \item for all $t\in T$, we have $\kappa t \in H'$; and
        \item for all $x,y\in I$ distinct, there exists some $s\in S$ such that $c(xy)=s+\cC'(x)+\cC'(y)+H'$.
    \end{enumerate}  
\end{lemma}

\begin{proof}
First, if $|X|\leq d+d'$ then setting $I=T=\emptyset$ the lemma is vacuously true. Otherwise, fix sets $U=\{u_1,\ldots,u_{d-1}\}$ and $V=\{v_1,\ldots,v_{d'-1}\}$ in $X$ so that $M\coloneqq U\cap V =\{u_1,\ldots,u_{m}\}=\{v_1,\ldots,v_{m}\}$. Since $X$ does not contain a $(d', \lambda)$--gadget, it follows that for each $1\leq i<d'$ we have
\[|\{c(Uw'v_i):w'\in X\setminus (U\cup V)\}|\leq \lambda-1.\]
Denote by $g_0^i$ the most popular color in the set above, so that there exist at least $(|X|-(d+d'-m-2))/(\lambda-1)$ vertices $w'\in X\setminus (U\cup V)$ such that $c(Uw'v_i)=g_0^i$.

Similarly, note that for each $1\leq j<d$
\[|\{c(u_jw'V):w'\in X\setminus (U\cup V)\}|\leq \lambda-1\]
holds and analogously define $q_0^j$ to be the most popular color in the above set.

Then there exists a set of vertices $I_1\subset X$ with $|I_1|\geq (|X|-(d+d'-m-2))/(\lambda-1)^{d+d'-2}$ such that for all $w'\in I_1$, $i \in [d'-1]$ and $j \in [d-1]$, we have $c(Uw'v_i)=g_0^i$ and $c(u_jw'V)=q_0^j$. 

Further, let us fix a vertex $w\in I_1$ and observe that $c(Uw'w)$ and $c(ww'V)$ both take at most $\lambda-1$ colors across all $w' \in I_1\setminus \{w\}$, and so we can again restrict to the most popular colors (denoted by $g_1$ and $q_1$, respectively) to obtain a set $I\subset I_1\setminus\{w\}$ of size
\[|I|\geq \left(|X|-(d+d'-m-2)\right)/(\lambda-1)^{d+d'}-1\geq \left(|X|/(\lambda-1)^{d+d'}\right) -d-d'\] 
such that for all $w'\in I$, $c(Uw'w)=g_1$ and $c(ww'V)=q_1$. Further, let us define $$g_0\coloneqq \sum_{i< d'}g_0^i = c(w'V) + (d'-1)\cdot c(Uw')$$ and $$q_0\coloneqq \sum_{j< d} q_0^j = c(w'U) + (d-1)\cdot c(w'V).$$

Then observe that for all $w'\in I$, we have
$$
\begin{aligned}
c(ww')&=c(ww') +c(w'U) - c(w'U) - (d-1)\cdot c(w'V)+(d-1)\cdot c(w'V)\\
&=g_1-q_0+(d-1)\cdot c(w'V)\\
&=g_1-q_0+(d-1)\cdot(q_1-c(ww'))\\
&= g_1-q_0+(d-1)\cdot q_1-(d-1)\cdot c(ww')
\end{aligned}$$
and thus $$d\cdot c(ww')=g_1-q_0+(d-1)\cdot q_1,$$
which is a constant in $\Gamma/H'$ independent of the choice of $w'$. Similarly, we also get that $$d'\cdot c(ww') =q_1-g_0+(d'-1)\cdot g_1$$
which is again independent of $w'$.

Then, for distinct $w',w''\in I$ we have that $d\cdot (c(ww')-c(ww''))=d'\cdot (c(ww')-c(ww''))=H'$ (recalling that $c=c_1/H'$ takes values in $\Gamma/H'$). Now, since $\gcd(d,d',n)=\kappa$ it follows by B\'ezout's identity that \[\kappa\cdot c(ww')=\kappa\cdot c(ww'')\]
for all $w', w''\in I$.

Next, fixing some $v_0\in I$, we define for all $x\in I$ the new vertex coloring $\cC':I\to\Gamma$ by letting $\cC'(x)$ be an arbitrary representative of the coset $c(xw)-c(wv_0)\in \Gamma/H'$, and note that $\kappa \cdot \cC'(x)\in H'$ for all $x\in I$. We then define $T\coloneqq \{\cC'(x):x\in I\}\subset \Gamma$ and note it satisfies the required properties.

We now only have left to define $S$ and ensure that it and $\cC'$ satisfy the third condition in the lemma. We observe that $c(UxyV)=c(xy)+q_1+g_1-c(xw)-c(yw)$, for all $x\neq y\in I$ and so
\[c(xy)=c(UxyV) -q_1-g_1+\cC'(x)+\cC'(y)+2c(v_0w)+H'.\] 
Now, by our assumption that a $(d,d',m,\lambda)$--gadget does not exist, we see that $c(UxyV)$ can take at most $\lambda-1$ colors as $x, y$ range in $I$; call this set of attained colors $S'$ and extend it so it has cardinality exactly $\lambda-1$, i.e. we have $S'=\{s'_1, \dots, s'_{\lambda-1}\}\subset \Gamma/H'$.

Letting $\hat{s}= 2c(v_0w)-q_1-g_1\in \Gamma/H'$ constant, we then have that $c(xy)=\hat{s}+c(UxyV)+\cC'(x)+\cC'(y)$. Then for each $i\in [\lambda-1]$ take some arbitrary coset representative $s_i$ of $s_i' + \hat{s}$ (we recall that $s_i', \hat{s}\in \Gamma/H'$) and define $S=\{s_1,\ldots,s_{\lambda-1}\}\subset \Gamma$ and conclude that for all $x\neq y\in I$ there is some $s_i\in S$ with $s_i+H'=\hat{s}+c(UxyV)$ and so  $c(xy)=s_i+\cC'(x)+\cC'(y)+H'$.
\end{proof}

\subsection{If Phase 1 fails}\label{phase_1}

Suppose in Phase 0 (\Cref{phase0}) we chose a $\kappa$--well-behaved tuple $(R,\Gamma,T_0,\cC_0,s)$ with size $\sigma=|\Gamma|\geq n/\alpha$, and so we ran \nameref{sub1} with parameters $(R_1,\Gamma,c,0_{id})$, where $\cC_0$ is identically equal to $0=0_{\Gamma_0}$ on $R_1$. Suppose our algorithm gets stuck on some iteration that is still in Phase 1, i.e. we have reached some vertex set $X\subseteq R_i$ for some $i$ and quotient group $\Gamma/H'_i$ and cannot realize a blueprint pair of type $(d,d',m)$. As $\cC_0$ is constant on $X$ and by \Cref{vertex_count_thm} $|X|>8\Delta^2n>2\Delta(\lambda-2+\Delta)$, so by \Cref{prelim_phase_arg}, there is no $(d,\beta)$--, $(d',\beta)$--, or $(d,d',m,\beta)$--gadget in the vertex set $X$ under the edge coloring $c_i= c_1/H'_i=(c_0-s)/H'_i$ and constant vertex coloring $\cC\equiv 0_{\Gamma_0}$. For ease of notation, in this subsection we drop the subscripts and refer to $H_i'$ and $c_i$ as $H'$ and $c$ respectively.

Then because the vertex coloring in Phase 1 is identically equal to $0_{\Gamma_0}$, and all blueprint pairs in $K$ satisfy $\gcd(d, d', n)=\kappa$ (\Cref{K_partition}), we may apply \Cref{core_phase_arg} to find sets $I\subset X$, $S=\{s_1,\ldots,s_{\beta-1}\}\subset \Gamma$, $T\subset \Gamma$, and a map $\cC:I\to T$ such that: 
    \begin{enumerate}
        \item $|I|\geq |X|/(\beta-1)^{d+d'} - d-d'$;
        \item for all $t\in T$, we have $\kappa t \in H'$; and
        \item for all $x,y\in I$ distinct, there exists some $s_i\in S$ such that $c(xy)=s_i+\cC(x)+\cC(y) + H'$.
    \end{enumerate}
    Moreover, arguing as in \Cref{T_size}, we may further assume that $|T|\leq |\Gamma/H'|$.

Finally, define $S_0=\{s_i-s_1:s_i\in S\}$ and $H''=H' + \langle S_0\rangle$.
\begin{claim*}$(I,H'',T,\cC,s+s_1)$ is $\kappa$--well-behaved and has size $|H''|<\sigma$.
\end{claim*}
It is clear that this claim leads to a contradiction, so that, in fact, our algorithm never fails in Phase 1.

\begin{subproof}[Proof of claim]
Recall the definition of a $\kappa$--well-behaved tuple (\Cref{well-behaved}). We have already shown \cref{C_correct} since for all $x\neq y\in I$ we have $c_0(xy)+H''=s+c_1(xy)+H''=s+s_i-s_1+s_1+\cC(x)+\cC(y)+\langle S_0\rangle +H'=s+s_1+\cC(x)+\cC(y)+H''$.
We also have \cref{T_order} because we have $\kappa \cC(x)\in H'\subset H''$ for all $x\in I$.

It remains to establish \cref{R_size} and that $|H''|<|\Gamma|$.

For this we want to argue as follows: either $\Gamma_0$ has at most a constant number of elements of order $\kappa$ (and so $|T|$ is bounded by some constant), or, otherwise, $|\langle S_0\rangle|$ is bounded from above by some constant.
For this, we need a particular algebraic result regarding the structure of $\Gamma_0$. Intuitively, if $\Gamma_0\cong \mathbb{Z}_\kappa \times \cdots  \times \mathbb{Z}_\kappa$, then $\Gamma_0$ has $n$ elements of order $\kappa$ but $|\langle S_0\rangle |\leq \kappa^{|S_0|}=\kappa^{\beta-1}$. On the other hand, if instead $\Gamma_0\cong \mathbb{Z}_{\kappa^t}$, then it's possible that $|\langle S_0\rangle|=n$ but there are only $\kappa$ elements of order $\kappa$. The precise result is stated formally below; we give its proof in \Cref{sec.pfalg}.

\begin{lemma}\label{algebra2}
    Let $\Gamma$ be a finite abelian group with $\cK$ elements of order $\kappa$. Let $X\subset \Gamma$ such that $|X|=x$. Then $|\langle X \rangle|\leq |\Gamma|\cdot \frac{\kappa^{x}}{\cK}$.
\end{lemma}

First suppose that there are strictly more than $\kappa^{2\beta}$ elements of order $\kappa$ in $\Gamma/H'$ (note this implies that $\kappa>1$ since there is only one element of order $1$). Since we are still in Phase 1, we know that $|\Gamma/H'|\geq n/\alpha$, and so $|H'|\leq\alpha$. By \Cref{algebra2}, it holds that 
\beq{bdH''}|H''|\leq |H'|\cdot |\langle S_0\rangle|\leq \alpha \cdot n\cdot \kappa^{\beta-1-2\beta} \leq \alpha \cdot n\cdot \kappa^{-2\alpha-1} \leq n\cdot (2^{-\alpha})^2\cdot \alpha/2 <  n/(2\alpha) \le |\Gamma|/2.\enq

Since $|H''|\leq n/\alpha$, to establish \cref{R_size} in \Cref{well-behaved} we will show that \[|I|>|R_0|/(\alpha\beta^{2\Delta}) - 14\Delta^3 (n-|H''|).\]
By \Cref{vertex_step} and  \Cref{vertex_jump} we have 
\begin{align*}
    |I|\geq |X|/(\beta-1)^{d_0+d_1} -d -d'
    > \big(|R_1|-7\Delta^3|\Gamma|(1-1/|H'|) -7\Delta^3(|\Gamma/H'|)\big)/\beta^{2\Delta} = (|R_1|-7\Delta^3|\Gamma|)/\beta^{2\Delta} .
\end{align*}
By \Cref{R'R_size} we have $|R_1| >|R|/\alpha$, so we can write 
\begin{align*}
    |I|& \ge (|R|/\alpha - 7\Delta^3|\Gamma|)/\beta^{2\Delta}= (|R| - 7\Delta^3\alpha|\Gamma|)/\alpha\beta^{2\Delta},
\end{align*}
and by \cref{R_size} in \Cref{well-behaved} we have
\begin{align*}
    |I| &>(|R_0| -14\Delta^3 (n-|\Gamma|) - 7\Delta^3\alpha|\Gamma|)/\alpha\beta^{2\Delta}
\\&> |R_0|/\alpha\beta^{2\Delta} - 14\Delta^3(n-|\Gamma|)-|\Gamma|
\\& = |R_0|/\alpha\beta^{2\Delta} - 14\Delta^3\left(n-\frac{14\Delta^3-1}{14\Delta^3}|\Gamma|\right)
\\ &> |R_0|/\alpha\beta^{2\Delta} - 14\Delta^3(n-|H''|),
\end{align*}
where the last inequality holds by \cref{bdH''}. Thus, we have shown \cref{R_size} and thus the desired claim in this case.

Suppose instead that there are at most $\kappa^{2\beta}$ elements of order $\kappa$ in $\Gamma/H'$, so that $|T|\leq \kappa^{2\beta}$. Then the largest color class $\cC_i\coloneqq \{x\in I:\cC(x)=i\}$ (call it $\cC_{\text{max}}$) has size at least $|I|/\kappa^{2\beta}$, and on this set, we have that $c=c_1/H'$ takes at most $|S|<\beta$ colors, which means that $c_1:E(\cC_{\text{max}})\to \Gamma$ takes at most $\alpha\beta$ colors since $|H'|\leq \alpha$. Thus, we have a set $\cC_{\text{max}}$ whose edges are colored by at most $\alpha\beta$ colors.
Thus, since $$\begin{aligned}
|\cC_{\text{max}}|&\ge |I|/\kappa^{2\beta}\\
&\geq |R_0|/\alpha\beta^{2\Delta}\kappa^{2\beta} - 14\Delta^3n\\
&\ge \frac{10\cdot 20^{2\Delta}\cdot \Delta^{41\Delta^6}}{10\Delta^6\cdot (20\Delta^6)^{2\Delta}\cdot \kappa^{40\Delta^6}}\cdot C'(\alpha\beta,\Delta) \cdot n - 14\Delta^3n\\
&\ge \Delta^{\Delta^6-6-12\Delta}\cdot C'(\alpha\beta,\Delta) \cdot n - 14\Delta^3n \\
&\ge C'(\alpha\beta,\Delta)\cdot n,
\end{aligned}$$ by \Cref{ramseylinear} it follows that the vertex set $\cC_{\text{max}}$ contains a copy of $G$ that is monochromatic under the edge-coloring $c_1$. However, recall that $c_0(xy)=c(xy)+s+\cC_0(x)+\cC_0(y)=c(xy)+s$ holds for all $x, y\in\cC_{\text{max}}$, and so this copy of $G$ is also monochromatic under our original edge-coloring $c_0$. This is a contradiction to the assumption that there is no such copy in $R_0$. This completes the proof of the claim, and hence we have shown that \nameref{Amain} can not fail to realize a blueprint pair when $\lambda = \beta$.
\end{subproof}

\begin{remark}
    This is where the value of the constant $C$ in \Cref{main} comes from. Thus, if it were possible to avoid going into Phase 1 entirely, and hence avoid appealing to \Cref{ramseylinear}, then the constant would be much smaller, and would not depend at all on $C'(\alpha\beta,\Delta)$.
\end{remark}

\subsection{If Phase 2 fails}\label{phase_2}

Suppose in Phase 0 (\Cref{phase0}) we chose a $\kappa$--well-behaved tuple $(R,\Gamma,T,\cC,s)$ with size $\sigma=|\Gamma|$, and that in our running of \nameref{sub1} with parameters $(R_1,\Gamma,c,\cC)$ we get stuck on some iteration $i$ in Phase 2, i.e. we have reached some vertex set $X\subset R_i$ and some nontrivial quotient group $\Gamma/H'_i$ and edge-coloring $c_i=c_1/H'_i$ and a blueprint pair $(\Upsilon_1, \Upsilon_2)\in K_i$ of type $(d, d', m)$ for which we cannot find a realization. Similarly to \Cref{phase_1}, we drop the $i$ subscript and refer to the relevant objects as $c$ and $H'$, respectively. Then, since we know that \nameref{Amain} ensures that each color class in $X$ has size at least $2\Delta^2=2\Delta(\lambda-2+\Delta)$,  by \Cref{prelim_phase_arg} we know that there is no $(d,2)$--, $(d',2)$--, or $(d,d',m,2)$--gadget in $X$. Recall also that every blueprint pair in $K$ satisfies $\gcd(d,d',n)=\kappa$. 

We will find suitable $I\subset X$, $T'\subset \Gamma_0$, $\cC':I\to T',$ and $s'\in \Gamma_0$ so that $(I,H',T',\cC',s')$ is $\kappa$--well-behaved and $\sigma(I,H',T',\cC',s')$ $<\sigma(R,\Gamma,T,\cC,s)=\sigma$, thus giving a contradiction.

We begin by finding a suitable set $T'\subset \Gamma$ and vertex coloring $\cC^*$ (which will \emph{not} be our final coloring $\cC'$, to be defined later). Define $\cC_t(X) = \{x\in X: \cC(x)=t\}$. Then for each $t\in T$ we can apply \Cref{core_phase_arg} with the set being $\cC_t(X)$ and $\lambda=2$ to obtain sets $I_t \subset \cC_t(X)$ and $T_t\subset \Gamma$, along with an element $s_t\in \Gamma$ and a map $\cC^*_t:I_t\to T_t$ such that:

\begin{enumerate}
    \item\label{Ibd} $|I_t|\geq |\cC_t(X)|-2\Delta$;
    \item for all $i\in T_t$, we have $\kappa i \in H'$; and
    \item for all $x,y\in I_t$ distinct, we have $c(xy)=s_t+\cC^*_t(x)+\cC^*_t(y)+H'$.
\end{enumerate}

But since the sets $\{\cC_t(X): t\in T\}$ are pairwise disjoint, we can find a map $\cC^*:I\to \Gamma$, where 
\beq{Idef}I\coloneqq \bigcup_{t\in T:|I_t|\geq 2\Delta^2} I_t,\enq
which extends each $\cC^*_t$ with $t\in T$ satisfying $|I_t|\geq 2\Delta^2$.

We now bound the size of $I$. Since the sets $\{\cC_t(X): t\in T\}$ partition $X$, using \cref{Ibd} above we have \beq{sizeI}|I|\geq |X|-(2\Delta^2+2\Delta)|T|.\enq

Now define $\cC_t\coloneqq \{x\in I: \cC(x)=t\}$ and $\cC_{t,i}\coloneqq\{x\in \cC_t: \cC^*(x)=i\}$ for each $t\in T$ and $i\in \Gamma$.

\begin{figure}[htbp]
  \begin{center}
  \begin{tikzpicture}[
  v2/.style={fill=black,minimum size=4pt,ellipse,inner sep=1pt},
  node distance=1.5cm,scale=0.5]

    \draw (-5,0) ellipse [x radius=2.5cm, y radius=4cm];
  \draw (-5,2.5) circle [radius=1.3cm];
 \draw (-5,0) circle [radius=0.8cm];
 \node[v2] (a2) at (-5,-2.5){};
  \node[font=\scriptsize,below] at (a2) {$a_2\in \cC_{a,k}$};
 \node[v2] (a1) at (-5,2.5){};
  \node[font=\scriptsize,above] at (a1) {$a_1$};
  \node[font=\scriptsize]  at (-5,0){$V$}; 
 \draw (a1)--(-4.242,0.256);
 \draw (a1)--(-5.758,0.256);
 \draw (a2)--(-4.242,-0.256);
 \draw (a2)--(-5.758,-0.256);
 
  \draw (5,0) ellipse [x radius=2.5cm, y radius=4cm];
  \draw (5,2.5) circle [radius=1.3cm];
 \draw (5,0) circle [radius=0.8cm];
 \node[v2] (b2) at (5,-2.5){};
  \node[font=\scriptsize,below] at (b2) {$b_2\in \cC_{b,\ell}$};
 \node[v2] (b1) at (5,2.5){};
  \node[font=\scriptsize,above] at (b1) {$b_1$};
  \node[font=\scriptsize]  at (5,0){$W$};
 \draw  (b1)--(4.242,0.256);
 \draw  (b1)--(5.758,0.256);
 \draw (b2)--(4.242,-0.256);
 \draw  (b2)--(5.758,-0.256);
 
 \draw (a1)--(b1);
  \draw (a2)--(b1);
    \draw (a2)--(b2);
  
   \node[font=\scriptsize]  at (-7,0){$\cC_a$}; 
   \node[font=\scriptsize]  at (7,0){$\cC_b$};
  \node[font=\scriptsize] at (-5.8,2.5){$\cC_{a,i}$};
  \node[font=\scriptsize] at (5.8,2.5){$\cC_{b,j}$};
 
\end{tikzpicture}
\end{center}
\caption{An example of stars described in  \Cref{lem:Xab}.}\label{fig:Xab}
\end{figure}
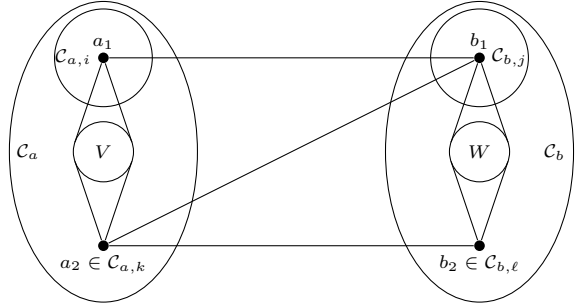
\begin{lemma}\label{lem:Xab}
    Given $a,b\in T$, there exists $s_{a,b}\in \Gamma$ such that for all $i\in T_a,j\in T_b$, and for all $a_1\in \cC_{a,i}$ and $b_1\in \cC_{b,j}$, we have
    $$c(a_1b_1) = i+ j + s_{a,b}+H'.$$

    In particular, we have $s_{a, a}=s_a$, where $s_a$ is as given by the application of \Cref{core_phase_arg} above.
\end{lemma}

\begin{proof}
Given $a, b\in T$, note that either we have $\cC_a=\emptyset$ or $\cC_b=\emptyset$, in which case the statement is vacuously true, or we have $|\cC_a|, |\cC_b|>2\Delta^2$, which we henceforth assume.

Pick arbitrary $a_2\in \cC_a$ and $b_2\in \cC_b$ with $a_2\neq b_2$, say with $a_2\in \cC_{a,k}$ and $b_2\in \cC_{b,\ell}$ for some $k, \ell\in\Gamma$. Then take any $V\subset \cC_a$ and $W\subset \cC_b$, both of size $d-1$, such that $V\cup W$ is disjoint from $\{a_1,a_2,b_1,b_2\}$. Define $\cC^*(V) \coloneqq  \sum_{v\in V} \cC^*(v)$ and $\cC^*(W) \coloneqq  \sum_{w\in W} \cC^*(w)$. 
  
Now consider the stars $Va_1b_1$ and $Va_2b_1$ (see \Cref{fig:Xab}), and recall our assumption that there is no $(d, 2)$--gadget under coloring $c$ in $X$. Thus we have $c(Va_1b_1)+d\cdot \cC(a_1)=c(Va_2b_1)+d\cdot \cC(a_2)$. Note that for each $v\in V$, we have $c(va_1)=s_a+i+\cC^*(v)+H'$ and $c(va_2)=s_a+k+\cC^*(v)+H'$. Since $\cC(a_1)=\cC(a_2)=a$, it holds that  
$$(d-1)\cdot s_a+(d-1)\cdot i+\cC^*(V)+c(a_1b_1)=(d-1)\cdot s_a+(d-1)\cdot k+\cC^*(V)+c(a_2b_1)$$
and so $$c(a_1b_1)=(d-1)\cdot (k-i)+c(a_2b_1)$$ and thus, recalling that $\kappa$ divides $d$ and that each $i\in T_a$ satisfies $\kappa\cdot i \in H'$, we get: $$c(a_1b_1)=i-k+c(a_2b_1).$$

By an identical argument (using the stars $Wb_1a_2$ and $Wb_2a_2$), we get that $$c(b_1a_2) = j-\ell+c(b_2a_2)$$ and so $$c(a_1b_1)=i+j-k-\ell+c(a_2b_2).$$ Thus $c(a_1b_1)-i-j=c(a_2b_2)-k-\ell$ is an element of $\Gamma/H'$ that depends only on $a$ and $b$, say $s_{a,b}+H'$ for some $s_{a, b}\in\Gamma$. Clearly, $i+j+s_{a,b}+H'=c(a_1b_1)$ as desired.
\end{proof}

Recall that the original $\kappa$--well-behaved tuple found in \Cref{phase0} is $(R,\Gamma,T,\cC,s)$ and that we are aiming to find some $\kappa$--well-behaved $(I,H',T',\cC',s')$ with a smaller size in order to reach a contradiction. So far, we have found $H'$ and $I$.

Fix some $r\in T$ so that $\cC_r$ is non-empty and define

\[t_{a,i}'\coloneqq a + i + s_{r,a} - s_r - r\]
for all $a\in T$ and $i \in T_a$. Note that $t_{r,i}' = r + i + s_r-s_r-r=i$. Then define
\beq{T'def}T' \coloneqq \{t_{a,i}': a\in T, i\in T_a, \cC_{a, i}\neq\emptyset\}.\enq

Recall that $c=c_1/H'$ and $c_1(xy)=c_0(xy)-s-\cC(x)-\cC(y)$.

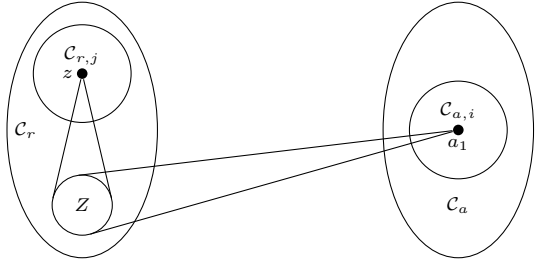
\begin{figure}[htbp]
  \begin{center}
  \begin{tikzpicture}[
  v2/.style={fill=black,minimum size=4pt,ellipse,inner sep=1pt},
  node distance=1.5cm,scale=0.5]

    \draw (-5,0) ellipse [x radius=2cm, y radius=3.4cm];
  \draw (-5,1.5) circle [radius=1.3cm];
 \draw (-5,-2) circle [radius=0.8cm];
 \node[v2] (a1) at (-5,1.5){};
  \node[font=\scriptsize,above] at (a1) {$\cC_{r,j}$};
   \node[font=\scriptsize,left] at (a1) {$z$};
  \node[font=\scriptsize]  at (-5,-2){$Z$}; 
\draw (a1)--(-4.221,-1.817);
 \draw (a1)--(-5.779,-1.817);

      \draw (5,0) ellipse [x radius=2cm, y radius=3.4cm];
    \draw (5,0) circle [radius=1.3cm];

   \node[font=\scriptsize]  at (-6.5,0){$\cC_r$}; 
\node[font=\scriptsize]  at (5,-2){$\cC_a$};
  \node[v2]  (a2) at (5,0){};
\node[font=\scriptsize,below] at (a2){$a_1$};
\node[font=\scriptsize,above] at (a2){$\cC_{a,i}$};
   
   \draw (a2)--(-5.095,-1.206);
 \draw (a2)-- (-4.782,-2.770);
 
\end{tikzpicture}
\end{center}
\caption{An example of stars in \Cref{prot2_T_order}.}\label{fig:tai}
\end{figure}
\begin{lemma}\label{prot2_T_order}
    We have $\kappa\cdot t' \in H'$ for all $t' \in T'$.
\end{lemma}

\begin{proof}
Given some $t'=t'_{a, i}\in T'$ for some $a\in T$ and $i\in T_a$, pick an arbitrary $z \in \cC_{r,j}$ (for some $j\in\Gamma$ where $\cC_{r,j}$ is non-empty),  $Z\subset \cC_r\setminus \{z\}$ with $|Z|=d$, and $a_1\in \cC_{a,i}$. Let $\cC^*(Z)=\sum_{z\in Z} \cC^*(z)$. 

Now compare the stars $Zz$ and $Za_1$ (see \Cref{fig:tai}), recalling our assumption that there is no $(d, 2)$--gadget under coloring $c$ in $X$, so $c(Zz)+d\cdot\cC(z)=c(Za_1)+d\cdot \cC(a_1)$. By \Cref{lem:Xab}, it holds that $c(Zz)= d\cdot s_r + \cC^*(Z)+d\cdot j+H'$ and $c(Za_1)=d\cdot s_{r,a}+\cC^*(Z)+d\cdot i+H'$. This yields that
\[\begin{aligned}
    d\cdot s_r + \cC^*(Z)+d\cdot j + d\cdot r+H'&=c(Zz) + d\cdot r \\
    &=c(Za_1) + d\cdot a = d\cdot s_{r,a}+\cC^*(Z)+d\cdot i+d\cdot a+H',
\end{aligned}\]
and further, since $d\cdot i, d\cdot j\in H'$, we get 
\[d\cdot (s_{r,a} +a -s_r-r)\in H'.\] 
Similarly, recall that there are no $(d', 2)$--gadgets in $X$ (with $\gcd(d, d')=\kappa'$), and by arguing as above we also obtain that
$$d'\cdot (s_{r,a}+a-s_r-r)\in H'$$ 
so by B\'ezout's identity, we get that 
$$\kappa'\cdot (s_{r,a}+a-s_r-r)\in H'$$ 
and so
$$\kappa\cdot (s_{r,a}+a-s_r-r)\in H'.$$

But since $\kappa i\in H'$, we have that
$$\kappa\cdot  t'_{a,i} =  \kappa\cdot (s_{r,a}+a-s_r-r)+\kappa\cdot i\in H',$$ 
as desired.
\end{proof}
We are finally in a position to define our vertex coloring. Define $\cC':I \to T'$ as
\beq{defC"}\cC'(x)\coloneqq t'_{\cC(x),\cC^*(x)}\enq
and
\beq{s'def}s'\coloneqq s+s_r+2r\enq
where we recall $s_r$ was given by our application of \Cref{core_phase_arg} at the start of this section. We want to establish \cref{C_correct} of \Cref{well-behaved}.

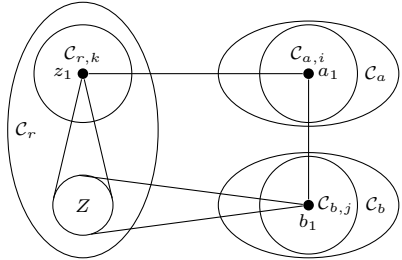
\begin{figure}[htbp]
  \begin{center}
  \begin{tikzpicture}[
  v2/.style={fill=black,minimum size=4pt,ellipse,inner sep=1pt},
  node distance=1.5cm,scale=0.5]

    \draw (-5,0) ellipse [x radius=2cm, y radius=3.4cm];
    \draw (-5,1.5) circle [radius=1.3cm];
 \draw (-5,-2) circle [radius=0.8cm];
 \node[v2] (a1) at (-5,1.5){};
  \node[font=\scriptsize,above] at (a1) {$\cC_{r,k}$};
   \node[font=\scriptsize,left] at (a1) {$z_1$};
  \node[font=\scriptsize]  at (-5,-2){$Z$}; 
\draw (a1)--(-4.221,-1.817);
 \draw (a1)--(-5.779,-1.817);
  
      \draw (1,1.5) ellipse [x radius=2.4cm, y radius=1.4cm];
    \draw (1,1.5) circle [radius=1.3cm];

    \draw (1,-2) ellipse [x radius=2.4cm, y radius=1.4cm];
    \draw (1,-2) circle [radius=1.3cm];

   \node[font=\scriptsize]  at (-6.5,0){$\cC_r$}; 
\node[font=\scriptsize]  at (2.8,1.5){$\cC_a$};
  \node[v2]  (a2) at (1,1.5){};
\node[font=\scriptsize,right] at (a2){$a_1$};
\node[font=\scriptsize,above] at (a2){$\cC_{a,i}$};

\node[font=\scriptsize]  at (2.8,-2){$\cC_b$};
  \node[v2]  (b2) at (1,-2){};
\node[font=\scriptsize,below] at (b2){$b_1$};
\node[font=\scriptsize,right] at (b2){$\cC_{b,j}$};

   \draw (b2) -- (a2);
   \draw (a1) -- (a2);
   \draw (b2)--(-5.095,-1.206);
 \draw (b2)-- (-4.782,-2.770);

\end{tikzpicture}
\end{center}
\caption{An example of stars in \Cref{lem_C_correct}.}\label{fig:CC}
\end{figure}
\begin{lemma}\label{lem_C_correct}
    Given $x\in \cC_{a,i}$ and $y\in \cC_{b,j}$ for $a, b\in T$, $i\in T_a$, and $j\in T_b$, we have 
    \[c_0(xy)+H' = s' +\cC'(x) +\cC'(y)+H'.\]
\end{lemma}
\begin{proof}
First, note that given such $x, y$ we have $\cC'(x)=t'_{a, i}$ and $\cC'(y)=t'_{b, j}$. Moreover, recall from the fact that the tuple $(R, \Gamma, T, \cC, s)$ is $\kappa$-well behaved, $c=c_1/H'$ and \Cref{lem:Xab} that \[c_0(xy)+H' = s_{a,b} + i + j +s + a + b + H'.\]

Now, since $d\cdot t'_{a,i} \in H'$, the statement of the lemma is equivalent to showing the quantity above equals $s' +t'_{a,i}-(d-1)\cdot t'_{b,j}+H'$. Equivalently, in terms of $s_{a,b}$, we want the following:
$$\begin{aligned}
    H'+s_{a,b}&=H'-i-j-s-a-b+s+s_r+2r+a+i+s_{r,a}-s_r-r-(d-1)\cdot (b+j+s_{r,b}-s_r-r)\\
    &=H'+d\cdot r-d\cdot b+s_{r,a}-(d-1)\cdot (s_{r,b}-s_r)-d\cdot j\\
    &=H'+d\cdot r-d\cdot b+s_{r,a}-(d-1)\cdot (s_{r,b}-s_r).
\end{aligned}$$
But now consider $ z_1 \in \cC_{r,k}$, $Z\subset \cC_r\setminus \{z\}$ with $|Z|=d-1$, $a_1\in \cC_{a,i}$, and $b_1\in \cC_{b,j}$. By comparing the stars $Zb_1a_1$ and $Zza_1$, and recalling our assumption that there is no $(d, 2)$--gadget under coloring $c$ in $X$,  we know that  $d\cdot b + c(Zb_1a_1) 
 = d\cdot r+ c(Zz_1a_1)$. This yields 
\begin{align*}
d\cdot b + (d-1)\cdot s_{r,b}+ d\cdot j+\cC^*(Z)+ s_{a,b} + i+H'
&= d\cdot b + c(Zb_1a_1)\\
&= d\cdot r+ c(Zz_1a_1) \\
&= d\cdot r + (d-1)\cdot s_r+d\cdot k+\cC^*(Z)+s_{r,a} +i+H',
\end{align*}
and so, recalling that $d\cdot j+H'=d\cdot k+H'=H'$, 
$$
\begin{aligned}H'+s_{a,b}&=H'-(d-1)s_{r,b} - d\cdot b+   (d-1)\cdot s_r+s_{r,a} + d\cdot r \\
&=H'+d\cdot r-d\cdot b+s_{r,a}-(d-1)\cdot s_{r,b}  +   (d-1)\cdot s_r,
\end{aligned}$$ as desired.
\end{proof}

Let us take stock of our progress in this section. We started with a $\kappa$--well-behaved tuple $(R, \Gamma, T, \cC, s)$, assumed that there is some iteration $i$ in Phase 2 of \nameref{sub1} for which we fail to find a suitable realization, and constructed a tuple $(I, H', T', \cC', s')$ (with $I$ defined in \cref{Idef}, $H'=H'_i$ given by our algorithm when it failed, $T'$ given in \cref{T'def}, $\cC'$ in \cref{defC"} and $s'$ in \cref{s'def}, respectively).

\begin{lemma}
The tuple $(I,H',T',\cC',s')$ is $\kappa$--well-behaved.
\end{lemma}

Note the above lemma immediately yields a contradiction because $|H'|<|\Gamma|$ (as $\Gamma/H'$ is non-trivial). 

\begin{proof}
First, \Cref{prot2_T_order} gives us \cref{T_order}. Further, \Cref{lem_C_correct} gives us \cref{C_correct}. To show \cref{R_size}, by \Cref{vertex_step} and \Cref{vertex_jump}, we have $|X|\geq |R_1| - 7\Delta^3(|\Gamma|-|\Gamma|/|H'|) - 7\Delta^3(|\Gamma/H'|)= |R_1| - 7\Delta^3|\Gamma|$ which together with \cref{sizeI} gives us
\begin{align*}
    |I|>|X|-(2\Delta^2+2\Delta)|T|&\geq |R_1| - 7\Delta^3|\Gamma|-4\Delta^2|T|.
\end{align*}
By \Cref{R'R_size} we have $|R_1|>|R| - 2\Delta^2 |T|$, giving
\begin{align*}
    |I|> |R| - 7\Delta^3|\Gamma| - 6\Delta^2|T|.
\end{align*}
By \cref{R_size} of \Cref{well-behaved}, $|R|\geq |R_0|/\alpha\beta^{2\Delta} - 14\Delta^3(n-|\Gamma|) -6\Delta^2 (n/|\Gamma|)$ and so we have
\begin{align*}
    |I| &> |R_0|/\alpha\beta^{2\Delta} - 14\Delta^3(n-|\Gamma|) - 7\Delta^3|\Gamma| -12\Delta^2 (n/|\Gamma|)
    \\&= |R_0|/\alpha\beta^{2\Delta} - 14\Delta^3 n+ 7\Delta^3|\Gamma|  -12\Delta^2(n/|\Gamma|).
    \end{align*}
Finally, because $H'$ is a proper subgroup of $\Gamma$, we have  $|\Gamma|\geq 2|H'|$ and
\begin{align*}
    |I|&\ge |R_0|/\alpha\beta^{2\Delta} - 14\Delta^3 n+ 14\Delta^3|H'| -6\Delta^2 (n/|H'|)\\
    &=|R_0|/\alpha\beta^{2\Delta} - 14\Delta^3 (n-|H'|) -6\Delta^2\ (n/|H'|)\label{equ:2-6}.
\end{align*}

Hence $(I,H',T',\cC',s')$ is indeed $\kappa$--well-behaved.
\end{proof}

This gives us the contradiction. Thus, \nameref{Amain} never fails to realize blueprints when $\lambda=2$ (Phase 2). Along with the analysis carried on in \Cref{phase_1}, this proves \cref{cond:realize} in \Cref{embG}, thereby completing the proof of \Cref{main}. 

\section{Conclusion and further work}\label{sec.conclusion}

In this paper we considered zero-sum Ramsey numbers in the case where $|\Gamma_0|=e(G)$. Note that \Cref{main} immediately gives the same bound $R(G, \Gamma)\leq C n$ whenever $|\Gamma|$ divides $e(G)$, where $n=e(G)$. However, more interestingly, we can observe that the value of $C$ should decrease in the case of $N\coloneqq  e(G)/|\Gamma|$ large. Let us give an informal overview of how the computations in our proof change in this case. Note that if $m_1<n$ (with $m_1$ as in \nameref{sub1}), then we run Phase 1 of the algorithm in our proof for fewer iterations. Further, recall how we computed $\alpha$ in \Cref{sec.blueprints}: we required $|K|\cdot 2\alpha \geq |\Gamma_0|$. In this case, we can divide both $\alpha$ and $\beta$ by $N$ (to get new $\alpha'$ and $\beta'$ to be used in the algorithm) which implies that the new constant needed is about $C\cdot \frac{C'(\alpha\beta/N^2,\Delta)}{C'(\alpha\beta,\Delta)}\cdot \frac{1}{N^{2\Delta+1}\cdot \Delta^{(2\beta)(N-1)/N}}$ (see the calculations at the end of \Cref{phase_1}). That is, we run our algorithm until $\alpha/N=1$, and so we never enter Phase 1 of the algorithm. This means that we do not require \Cref{ramseylinear}, and then the new constant that replaces $C$ is determined entirely by \Cref{vertex_jump} and is on the order of $\Delta^3$. We choose not to present all the formal details here.

It is natural to try and extend our result to other graph classes, and try to prove linear upper bounds for more general graphs than just those of bounded degree. Recall, however, that if $n=e(G)$ then $R(G, \mathbb{Z}_n)\geq R(G)$, and so we cannot expect to obtain linear zero-sum Ramsey numbers for graphs that have too large classical Ramsey numbers. A natural family to consider is that of graphs of bounded degeneracy, and here a celebrated result of Lee \cite{Lee17} shows that $R(G)$ is indeed linear for such graphs $G$. We expect the same to be true for $R(G, \Gamma_0)$. 
\begin{conjecture}\label{prob:degenerate}
    Let $d\geq 1$ be an integer. Then there exists a constant $C=C(d)$ such that for any graph $G$ with degeneracy at most $d$ and $n$ edges, and any finite abelian group $\Gamma_0$ of order dividing $n$, we have $R(G, \Gamma_0)\leq C\cdot n$.
\end{conjecture}

Note that the only point in our proof where we used the assumption of bounded degree is in \Cref{sec.blueprints}, where our construction of blueprint pairs required that these blueprint pairs each only use a small number of vertices. Therefore, if one could prove a statement akin to \Cref{K_partition} for $d$-degenerate graphs, a positive resolution to \Cref{prob:degenerate} would then follow using the same strategy as in our current work. We also note that $1$--degenerate graphs are precisely forests, and so a very recent result of Colucci and D'Emidio \cite{CD25} resolves Conjecture \ref{prob:degenerate} for $d=1$, $\Gamma_0=\mathbb{Z}_p$, and $v(G)$ sufficiently large.

Finally, we have studied zero-sum Ramsey numbers of bounded degree graphs under the assumption that the group is abelian. It is natural to ask a similar question for non-abelian groups, as in \cite{Caro92allgroups}.
  \begin{problem}\label{prob:nonabelian}
   Given an integer $\Delta\geq 2$, does there exist a constant $C(\Delta)$ such that for any graph $G$ with maximum degree $\Delta$ and $n$ edges and any finite group $\Gamma$ of order $n$ we have $R(G, \Gamma)\leq Cn$?
 \end{problem}

 The proof of \Cref{main} relies heavily on \Cref{Kneser}, which does not apply to non-abelian groups \cite{OlsonKneser}. Hence, new ideas would be needed to study the non-abelian case and we therefore think any result on non-abelian groups would be interesting, even in simpler cases such as for $G$ being a path or for $|G|=N\cdot |\Gamma|$ for some large $N$.

\section*{Acknowledgments}
This work began when the authors attended a workshop organized by Leo Versteegen at the London School of Economics. We would like to thank Leo for his organizing and the LSE for its hospitality. We would also like to thank Julia Böttcher, Anubhab Ghosal and Matthew Jenssen for feedback on a previous draft of this manuscript.

\subsubsection*{Competing interests statement}
Competing interests: The authors declare none.

\subsubsection*{Funding declaration}

Xiaopan Lian was supported by National Natural Science Foundation of China No. 12371351. Alexandru Malekshahian was supported by ERC Advanced Grant 883810. The funders had no role in study design, data collection and analysis, decision to publish, or preparation of the manuscript.

\bibliographystyle{plain}
\bibliography{bibliography}

\appendix

\section{Proof of Lemma \ref{algebra2}}\label{sec.pfalg}

The proof of Lemma \ref{algebra2} uses the following two facts, both of which are consequences of Theorem \ref{thm:bascigroupthm}.
\begin{fact}\label{fact:alg1}

Let $p$ be prime, and let $\Gamma = \ZZ_{p^{a_1}} \times \cdots  \times \ZZ_{p^{a_t}}$, where $a_i \in \NN$ for all $1 \le i \le t$. Let $\Gamma'$ be a subgroup of $\Gamma$. Then $\Gamma' \cong \ZZ_{p^{b_1}} \times \cdots  \times \ZZ_{p^{b_t}}$, where each $b_i \le a_i$ for all $1 \le i \le t$. 
\end{fact}

\begin{fact}[\!\! {\cite[Theorem 3.1]{min_gen_set}}]\label{fact:alg2} 

 Let $\Gamma$ be a finite abelian group with invariant factor decomposition given by $\Gamma = \mathbb{Z}_{m_1} \times\cdots\times \mathbb{Z}_{m_t}$,
    where $m_i \mid m_{i+1}$ for each $1 \le i \le t-1$. 
    Let $s$ be the smallest positive integer such that $\Gamma = \langle a_1,\ldots,a_s\rangle$ for some $\{a_1,\ldots,a_s\} \se \Gamma$. \\
    Then $s = t$. 
\end{fact}

\begin{proof}[Proof of Lemma \ref{algebra2}]
    Write $\kappa=\prod_{p\in S} p^{k_p}$ where $S$ is a set of prime divisors of $\kappa$.
 By \Cref{thm:bascigroupthm},  we assume that $  \Gamma=J\times \prod_{p\in S}\left( \mathbb{Z}_{p^{a_{(p,1)}}} \times \cdots  \times \mathbb{Z}_{p^{a_{(p,t_p)}}}\right)$, where $a_{(p,i)}\geq a_{(p,i+1)}$ for all $1 \le i<t_p$ and $J$ is some subgroup of $\Gamma$. Let $i_p$ be the last $i$ such that $a_{(p,i)}\geq k_p$.  Then   $$\cK = \prod_{p\in S} \left(\prod _{i:a_{(p,i)\geq k_p}}p^{k_p}\cdot \prod_{i:a_{(p,i)}<k_p}p^{a_{(p,i)}}\right)=\prod_{p\in S} \left(p^{k_pi_p}\cdot \prod_{t_p\geq i>i_p}p^{a_{(p,i)}}\right).$$  Thus, $$\kappa^x/\cK =\prod_{p\in S} \left(p^{(x-i_p)k_p}\cdot \prod_{t_p\geq i>i_p}p^{-a_{(p,i)}}\right).$$
    
    Further let $\Gamma_p=\mathbb{Z}_{p^{a_{(p,1)}}} \times \cdots  \times \mathbb{Z}_{p^{a_{(p,t_p)}}}$ and define the projection $\psi_p:\Gamma\to \Gamma_p$. Note that $$|J| = |\Gamma|\cdot \prod_{p\in S}p^{-\sum_{i=1}^{t_p} a_{p,i}}.$$
    It is clear that for any $X\subset \Gamma$ we have $$|\langle X \rangle| \leq |J| \cdot \prod_{p\in S}|\langle \psi_p(X)\rangle|=|\Gamma|\cdot \prod_{p\in S}p^{-\sum_{i=1}^{t_p} a_{(p,i)}} \cdot \prod_{p\in S}|\langle \psi_p(X)\rangle|.$$ Let $S'=\{p\in S: t_p\geq x\}$. Then we want to show for all $p\in S'$ that $$|\langle \psi_p(X)\rangle|\leq p^{\sum_{i= 1}^{x} a_{(p,i)}},$$
    as for $p\in S\setminus S'$ we instead simply have $$|\langle \psi_p(X)\rangle|\leq p^{\sum_{i=1}^{t_p}a_{(p,i)}}.$$

    Fixing $p\in S'$ let $a_i\coloneqq a_{(p,i)}$ and $t\coloneqq t_p$. As $\langle \psi_p(X) \rangle$ is a subgroup of $\Gamma_p = \mathbb{Z}_{p^{a_1}} \times \cdots  \times \mathbb{Z}_{p^{a_t}}$, by \Cref{fact:alg1}, we have that \begin{equation*}\label{eq:prime_decomp}
        \langle \psi_p(X) \rangle \cong \mathbb{Z}_{p^{b_1}} \times \cdots  \times \mathbb{Z}_{p^{b_t}}
    \end{equation*} where each $b_i \le a_i$. We note that $|\psi_p(X)| \le x$, so 
    \begin{equation*}\label{eq:non-prime_decomp}
        \langle \psi_p(X)\rangle \cong \mathbb{Z}_{m_1} \times\cdots\times \mathbb{Z}_{m_k},
    \end{equation*}
    where the $m_i$ are given by \Cref{thm:bascigroupthm}. Note that $k \le x$ by \Cref{fact:alg2}. However, as no two of $p^{b_1}, \ldots, p^{b_t}$ are coprime, we must have $k=t$ and $m_i = p^{b_i}$.  We have $|\langle \psi_p(X) \rangle| \le |\mathbb{Z}_{p^{a_1}} \times \cdots  \times \mathbb{Z}_{p^{a_x}}| = p^{\sum_{i=1}^xa_i}$, as required.

    Then $$  \begin{aligned}
       |\langle X\rangle| \frac{\cK}{\kappa^x}&\leq |\Gamma| \cdot \prod_{p\in S'} \left(p^{(i_p-x)k_p-\sum\limits_{i=x+1}^{t_p}a_{(p,i)}+\sum_{i=i_p+1}^{t_p}a_{(p_,i)}}\right)\cdot \prod_{p\in S\setminus S'}\left(p^{(i_p-x)k_p+\sum\limits_{i=i_p+1}^{t_p}a_{(p,i)}}\right)\\
        &=|\Gamma| \cdot \prod_{\substack{p\in S':\\ i_p\geq x}} \left(p^{(i_p-x)k_p-\sum\limits_{i=x+1}^{i_p}a_{(p,i)}}\right)\cdot \prod_{\substack{p\in S':\\ i_p<x}} \left(p^{(i_p-x)k_p+\sum\limits_{i=i_p+1}^{x}a_{(p,i)}}\right)\cdot   \prod_{p\in S\setminus S'}\left(p^{(i_p-x)k_p+\sum\limits_{i=i_p+1}^{t_p}a_{(p,i)}}\right) \\
        &\leq |\Gamma| \cdot \prod_{\substack{ p\in S':\\
        i_p\geq x}}\left(p^{(i_p-x)k_p-(i_p-x)k_p}\right)\cdot \prod_{\substack{p\in S':\\ i_p<x}} \left(p^{(i_p-x)k_p+(x-i_p)(k_p-1)}\right)\cdot \prod_{p\in S\setminus S'}\left(p^{(i_p-x)k_p+(t_p-i_p)(k_p-1)}\right)\\
        &= |\Gamma| \cdot \prod_{p\in S: i_p<x} p^{-x+i_p}\cdot \prod_{p\in S\setminus S'}p^{(t_p-x)k_p-t_p+i_p}\\
        &\leq |\Gamma|
         \end{aligned} $$ since for all $p\in S\setminus S'$, $t_p<x$. This concludes the proof.
\end{proof}
\end{document}